%% file: main.tex
\documentclass[groupedaddress,10pt,notitlepage,preprintnumbers]{revtex4-1}
\usepackage{amsmath}
\usepackage{amssymb}
\usepackage{amsthm}
\usepackage{mathtools}
\usepackage[dvipsnames]{xcolor}
\usepackage{csl}
\usepackage[toc,page]{appendix}

\input{headerdre.tex}

\newtheorem{assumption}{Assumption}[]
\newtheorem{theorem}{Theorem}
\newtheorem{lemma}{Lemma}
\newtheorem{corollary}{Corollary}
\numberwithin{lemma}{section}
\numberwithin{corollary}{section}
\numberwithin{theorem}{section}
\numberwithin{assumption}{section}
\numberwithin{equation}{section}

\makeatletter
\renewcommand{\p@subsection}{}
\renewcommand{\p@subsubsection}{}
\makeatother

\theoremstyle{remark}
\newtheorem*{remark}{Remark}

\allowdisplaybreaks[1]

\begin{document}

\title{An Explicit Probabilistic Derivation of Inflation in a Scalar Ensemble Kalman Filter for Finite Step, Finite Ensemble Convergence}

\cslauthor{Andrey A Popov and Adrian Sandu}
\cslyear{20}
\cslrevision{1}
\cslemail{apopov@vt.edu, sandu@cs.vt.edu}
\csltitlepage

\begin{abstract}
This paper uses a probabilistic approach to analyze the converge of an ensemble Kalman filter solution to an exact Kalman filter solution in the simplest possible setting, the scalar case, as it allows us to build upon a rich literature of scalar probability distributions and non-elementary functions. To this end we introduce the bare-bones Scalar Pedagogical Ensemble Kalman Filter (SPEnKF). We show that in the asymptotic case of ensemble size, the expected value of both the analysis mean and variance estimate of the SPEnKF converges to that of the true Kalman filter, and that the variances of both tend towards zero, at each time moment. We also show that the ensemble converges in probability in the complementary case, when the ensemble is finite, and time is taken to infinity. Moreover, we show that in the finite-ensemble, finite-time case, variance inflation and mean correction can be leveraged to coerce the SPEnKF converge to its scalar Kalman filter counterpart. We then apply this framework to analyze perturbed observations and explain why perturbed observations ensemble Kalman filters underperform their deterministic counterparts.
\end{abstract}

\author{Andrey A Popov}
\email[]{apopov@vt.edu}
\affiliation{Computational Science Laboratory, Department of Computer Science, Virginia Tech}

\author{Adrian Sandu}
\email[]{sandu@cs.vt.edu}
\affiliation{Computational Science Laboratory, Department of Computer Science, Virginia Tech}
\date{\today}

\maketitle 
\thispagestyle{empty} 

\section{Introduction}

This paper is concerned with explicitly proving the convergence of a scalar ensemble Kalman filter in three cases: finite step and asymptotic ensemble, finite ensemble and asymptotic step, and finite step and finite ensemble with optimal inflation.

Bayesian data assimilation~\cite{asch2016data,law2015data,evensen2009data,reich2015probabilistic}, in our view, is concerned with transforming our \textit{a priori} uncertainty about the state of a (often chaotic) dynamical system~\cite{strogatz2014nonlinear}, and our uncertainty about observations of some truth. Representatives of our uncertainty in this context are taken to be the distributions of some random variables, and an application of Bayesian inference would involve applying Bayes' rule in an exact manner. In the vast majority of cases this problem is intractable.

By the principle of maximum entropy~\cite{jaynes2003probability}, if an there exists an ensemble of states with support from all of real space, and there is no other prior information, then if the mean and the covariance of the ensemble are known, then the best distribution that we can prescribe to the ensemble to describe our uncertainty is Gaussian. The ensemble Kalman filter~\cite{evensen1994sequential,le2009large,kwiatkowski2015convergence,mandel2011convergence,butala2008asymptotic} (EnKF) abuses this notion by assuming that the first two statistical moments of the ensemble are good descriptions of the exact moments, and that no other moments are known. In this way, the EnKF attempts to utilize the Kalman filter~\cite{kalman1960new} (KF) framework by substituting exact moments for their ensemble-derived statistical estimates. The advantage of the EnKF is in both the utilization of ensemble propagation~\cite{kalnay2003atmospheric}., and in better estimation of model forecast uncertainty. As the transformations defined by the EnKF do not exactly solve the problem of Bayesian inference, unlike those of the KF, the EnKF is wrought with heuristic attempts to correct it. One such heuristic is inflation~\cite{anderson1999monte}, which is thought to separate the anomalies in space in order for the ensemble covariance estimate to not degenerate prematurely.

All previous attempts at the convergence of the EnKF have looked at the asymptotic case of a large ensemble~\cite{le2009large,kwiatkowski2015convergence}. To date there has not been a full comprehensive analysis of the EnKF in the case of a finite ensemble, and more importantly, in the case of finite steps. We do not claim to provide such an analysis in the general case, but instead in a simplified scalar case. Additionally we also aim to explicitly derive inflation as not a heuristic, but as a natural consequence to the EnKF analysis non-linear action on the ensemble.

The paper is organized as follows. In section~\ref{sec:bg} we review the Kalman filter, and an over-sampled square-root ensemble Kalman filter.
We then introduce the scalar Kalman filter in section~\ref{sec:skf} and prove that it correctly models the uncertainly in a scalar linear dynamical system.
We follow this with an introduction of a toy scalar ensemble Kalman filter meant for pedagogical purposes called the Scalar Pedagogical Ensemble Kalman Filter (SPEnKF) in section~\ref{sec:spenkf}.  
We derive explicit probability distributions~\cite{mood1974introduction}, explicit formulations of their moments, and explicitly describe the asymptotic behavior in both ensemble size and steps of such a formulation.
We then show that the distributions of the SPEnKF's mean and variance estimates degenerate to that of the scalar Kalman filter in the asymptotic ensemble case (subsection~\ref{subsec:ppfe}). 
We derive sequential step-wise variance inflation, and mean correction factors, such that when these factors are applied, the expected values of the mean and variance estimates of the SPEnKF are exactly the mean and variance estimate of the scalar Kalman filter (subsection~\ref{subsec:opti}). 
Moreover, we show that in the step limit that the finite-ensemble SPEnKF converges in probability, regardless of model behavior, to that of the Kalman filter (subsection~\ref{subsec:cfe}).
We then use our framework to look as to why EnKF with perturbed observations can potentially behave in a suboptimal manner (subsection~\ref{subsec:po}). Next, we provide a trivial multivariate extension ot the SPEnKF, and show that a form of localization can indeed reduce the need for an oversampled ensemble in section~\ref{sec:multivariate}. We end with some final thoughts in section~\ref{sec:conclusions}.

\section{Background}%
\label{sec:bg}

Consider the case of capturing our uncertainty about an unknown dynamical system, 
\begin{equation}
    \mxt_{i+1} = \!M^{\|t}_{i}(\mxt_{i}),
\end{equation}
that evolves a true state, $\mxt_{i}$, from step $i$ to step $i+1$.
Now, consider us having access to an imperfect model of this dynamical system, 
\begin{equation}
\label{eq:model}
X_{i+1} = \!M_{i}(X_{i}) + \xi_{i}, \quad i=0,2,\dots
\end{equation}
where the distribution of the random variables $\mx_{i}$ and $\mx_{i+1}$ represent our uncertainty about the true state at the respective steps, and the distribution of the random variable $\xi_{i}$ represents our uncertainty in the model propagation, commonly referred to as model error.

\begin{assumption}[Initial state]
\label{ass:state}
We assume that we have uncertainty about the initial state, and that this uncertainty is described by a normal distribution $\mx_0 \sim \mathcal{N}(\mxmean_0,\P_0)$.
\end{assumption}

\begin{assumption}[Model]
\label{ass:model}
We make the following simplifying assumptions:
\begin{enumerate}
\item The model \eqref{eq:model} is  linear, $\!M_{i}\coloneqq\M_{i}$, and 
\item The model error is an unbiased normal random variable, $\*\xi_{i} \sim \Ndist(0,\Q_{i})$.
\end{enumerate}
\end{assumption}

Similarly, the observations, $\my_{i}$, at step $i$ correspond to a transformation of the state of our system into a (usually lower dimensional) observation space, through an observation operator, $\!H_i$. Thus an observation at time $i$ can be obtained from
\begin{equation}
\label{eq:observations}
    \my_{i} = \!H_i(\mx_{i}) + \eta_{i},
\end{equation}
where similarly, the distribution of the random variable, $\eta_{i}$, represents our uncertainty in the observations, is commonly referred to as observation error, and is typically used to account for inaccuracies in our measurements.

\begin{assumption}[Observations]
\label{ass:observations}
We make the following simplifying assumptions:
\begin{enumerate}
\item The observation operator \eqref{eq:observations} is  linear, $\!H_i\coloneqq\H_i$, and 
\item The observation error is an unbiased normal random variable, $\eta_i \sim \Ndist(0,\R_i)$. 
\end{enumerate}
\end{assumption}

Under the stated Assumptions \ref{ass:state}, \ref{ass:model}, \ref{ass:observations}, and a perfect application of Bayes' rule, our uncertainty in the state of our system remains Gaussian at all times. The \textit{a priori} (forecast) probability distribution of the uncertainty in the state at the current step $i$ is $\Ndist(\mxmeanf_{i},\Pf_{i})$, and the \textit{a posteriori} (analysis) probability distribution of the uncertainty in the state at step $i$ is $\Ndist(\mxmeana_{i},\Pa_{i})$.

The {\it forecast step} propagates the mean and covariance of our uncertainty in the state through the model \eqref{eq:model} from step $i$ to $i+1$:
\begin{equation}
\label{eq:kff}
\begin{split}
    \mxmeanf_{i+1} &= \M_{i} \mxmeana_{i},\\
    \Pf_{i+1} &= \M_{i} \Pa_{i} \M_{i}^\intercal + \Q_{i},
    \end{split}
 \end{equation}
 where $\mxmeana_0 \coloneqq \mxmean_0$, and $\Pa_0 \coloneqq \P_0$.
 
The corresponding previous {\it analysis step} applies the canonical Kalman filter equations \cite{kalman1960new,sarkka2013bayesian}
\begin{equation} \begin{split}
    \begin{split}
    \mxmeana_{i} &= \mxmeanf_{i} - \K_{i}(\H_{i}\mxmeanf_{i} - \my_{i}),\\
    \Pa_{i} &= (\I - \K_{i}\H_{i})\Pf_{i},\\
    \K_{i} &= \Pf_{i}\H_{i}^\intercal{(\H_{i}\Pf_{i}\H_{i}^\intercal + \R_{i})}^{-1},
    \end{split}\label{eq:kfa}
\end{split} \end{equation}
to obtain the best linear unbiased estimate of our uncertainty in the state of a linear dynamical system under Gaussian error assumptions. Equation~\eqref{eq:kfa} calculates the \text{a posteriori} uncertainty from the prior information and the information described by the observations (and our uncertainty in them). 


The ensemble Kalman filter takes a Monte Carlo approach to represent the prior and posterior probability densities. 

The ensemble Kalman filter, instead of representing our uncertainties by the first two empirical moments of a normal distribution, attempts to represent our uncertainty by the first two statistical moments of an ensemble of samples.
One replaces the analytical Gaussian density defined by the mean and covariance with an empirical distribution defined by an ensemble of $N$ states, $\En{\mx} = [\mxs{1},\mxs{2}, \dots, \mxs{N}]$. The ensemble mean $\mxmean$ will now represent the mean estimate of the Kalman filter, and the sample covariance estimate will similarly represent the covariance estimate of the Kalman filter. Recall that a sample covariance $(1/(N-1)) \A \A^{\intercal}$ is calculated using the matrix of sample anomalies, $\A = \En{\mx} - \mxmean\,\1_N^\intercal$, which is the matrix of the differences between the ensemble members and the ensemble mean.

In our formulation of the ensemble Kalman filter, we ignore model error ($\*Q_i = \*0$ in \eqref{eq:model}), we set the observation error covariance matrix and the observation operator to be constant in time ($\*R_i = \*R$, $\*H_i = \*H$), and look at an ideal oversampled square-root filter, in which the covariance matrix estimates come from a distribution with finite variances. In a square-root filter\cite{tippett2003ensemble} the covariance is transported through the analysis step using a transformation of the probability distribution. This transformation is typically done on the ensemble anomalies. In a perfect square-root filter, with a linear model, both the mean and the anomalies can be completely decoupled from each other, thus we will take the anomalies to not be derived from the ensemble mean at all, thus getting an additional degree of freedom, making our statistical covariance estimate  $(1/N) \A \A^{\intercal}$ instead.

We write the mean propagation by the EnKF in a similar manner to that of the Kalman filter:
\begin{equation} \begin{split}
    \mxmeanf_{i+1} &= \M_{i} \mxmeana_{i},\\
    \mxmeana_{i+1} &= \mxmeanf_{i+1} - \Kh_{i+1}(\H\mxmeanf_{i+1} - \my_{i+1}).
\end{split}\label{eq:meansqrttransport} 
\end{equation}

We propagate the ensemble anomalies in a way that follows the linear structure of the base Kalman filter, through an approximate transport of distributions. Let $\Phf$ denote the ensemble estimate of the \textit{a priori} covariance matrix , and $\Pha$ denote the ensemble estimate of the \textit{a posteriori} covariance matrix. The EnKF propagated anomalies and covariances at the corresponding previous step $i$ take the form:
\begin{equation} \begin{split}
    \Af_{i} &= \M_{i-1} \Aa_{i-1},\\
    \Aa_{i} &= {\left(\Pha_{i}\right)}^{\frac{1}{2}}{\left(\Phf_{i}\right)}^{-\frac{1}{2}}\Af_{i},\\
    \Kh_{i} &= \Phf_{i}\*H^\intercal{(\H\Phf_{i}\H^\intercal + \R)}^{-1},\\
    \Phf_{i} &= \frac{1}{N}(\Af_{i} \AfT_{i}),\\
    \Pha_{i} &= (\I - \Kh_{i}\H)\Phf_{i}.
\end{split}\label{eq:anomalysqrttransport}
\end{equation}
%
%
Note that, as the anomalies undergo a non-linear transformation through equation~\eqref{eq:anomalysqrttransport}, the distribution of the ensemble-estimated analysis covariance matrix is not the (scaled) Wishart distribution.

\section{The Scalar Kalman Filter (SKF)}%
\label{sec:skf}

We focus on the analysis of a scalar Kalman filter as this allows us to obtain analytical results that are almost intractable in the multivariate case.

\subsection{Definition of the SKF}

\begin{assumption}[Perfect scalar model]\label{ass:pm}
Our model state comes from $x \in \mathbb{R}$, the linear model is exact (no model error), can vary at each step, and is non-trivial, $\*M_i\coloneqq m_i\not=0$.
\end{assumption}
We can think of this assumption as requiring that that truth is also propagated through the scalar linear model,
\begin{align}
    \xt_{i+1} = m_i \xt_i.
\end{align}

\begin{assumption}[Direct observation]
We observe our only component directly, $\*H \coloneqq h = 1$.
\end{assumption}

\begin{assumption}[Constant observation error]\label{ass:constobs}
 The distribution of the observation error will be taken to be the same at each step, and each observation, $y_i$, is to be drawn from a normal distribution with mean $\xt_i$ and variance $\R\coloneqq r>0$.
\end{assumption}

The filtering process starts with the initial values $\xf_0 \coloneqq x_0$, and $\pf_0 \coloneqq p_0$.

Our model propagation step (equation~\eqref{eq:kff} in the multivariate case) is:
\begin{equation}
\label{eq:skfvp}
  \begin{split}
  x_{i+1}^\|f &= m_i \xa_i,\\
  \pf_{i+1} &= m_i^2 \pa_i.
  \end{split}
\end{equation}
The corresponding analysis step (equation~\eqref{eq:kfa} in the multivariate case)  has the form:
\begin{equation}
\begin{split}
    \label{eq:skfa}
  \xa_i &= \xf_i + k_i(y_i - \xf_i),\\
  \pa_i &= (1 - k_i)\pf_i,\\
  k_i &= \frac{\pf_i}{\pf_i + r}.
  \end{split}
\end{equation}

Next we will prove some fundamental things about linear propagation in the scalar case.

\subsection{Properties of the SKF}

We first wish to analyze the propagation of variance through the filter. We will now prove that the only non-linear operation that happens to the variance is in the computation of the Kalman gain. Note that propagation of variance through the model is trivially linear from~\eqref{eq:skfvp}. We now prove that the analysis variance is a linear scaling of the Kalman gain.
\begin{lemma}\label{lem:par}
  The analysis variance at the $i$-th step is $\pa_i = r k_i$.
\end{lemma}
\begin{proof}
  We manipulate the variance analysis in equation~\eqref{eq:skfvp}:
\begin{equation*}
\begin{split}
    \pa_i &= (1-k_i)\pf_i = \left(1-\frac{\pf_i}{\pf_i + r}\right)\pf_i\\
            &= \left(\frac{\pf_i + r}{\pf_i + r} - \frac{\pf_i}{\pf_i + r}\right)\pf_i = \frac{r \pf_i}{\pf_i + r} = r k_i.
          \end{split}
\end{equation*}
\end{proof}
Moreover we can show that the Kalman gain at each step is a linear fractional function of the initial input variance.
\begin{lemma}\label{lem:kalmangain}
  The Kalman gain in the scalar Kalman filter at the $i$-th  step is 
\begin{equation*}
    k_i = \frac{\left(\prod_{j=0}^{i-1}m_j^2\right)p_0}{\left(\sum_{l=0}^{i}\prod_{j=0}^{l-1}m_j^2\right)p_0 + r}.
\end{equation*}
\end{lemma}
\begin{proof}
  The Kalman gain at step $0$ is clearly $k_0=\frac{p_0}{p_0+r}$, now we manipulate in typical inductive fashion.
  Assume that,
\begin{equation*}
\begin{split}
    k_{q-1} &= \frac{\left(\prod_{j=0}^{q-2}m_j^2\right)p_0}{\left(\sum_{l=0}^{q-1}\prod_{j=0}^{l-1}m_j^2\right)p_0 + r},\\
    \pa_{q-1} &= r\frac{\left(\prod_{j=0}^{q-2}m_j^2\right)p_0}{\left(\sum_{l=0}^{q-1}\prod_{j=0}^{l-1}m_j^2\right)p_0 + r},
  \end{split}
\end{equation*}
  then, by~\eqref{eq:skfvp} and Lemma~\ref{lem:kalmangain},
\begin{equation*}
\begin{split}
    \pf_{q} &= m_{q-1}^2 r\frac{\left(\prod_{j=0}^{q-2}m_j^2\right)p_0}{\left(\sum_{l=0}^{q-1}\prod_{j=0}^{l-1}m_j^2\right)p_0 + r} = r\frac{\left(\prod_{j=0}^{q-1}m_j^2\right)p_0}{\left(\sum_{l=0}^{q-1}\prod_{j=0}^{l-1}m_j^2\right)p_0 + r}\\
    k_q &= {\left[r\frac{\left(\prod_{j=0}^{q-1}m_j^2\right)p_0}{\left(\sum_{l=0}^{q-1}\prod_{j=0}^{l-1}m_j^2\right)p_0 + r}\right]}     {\left[r\frac{\left(\prod_{j=0}^{q-1}m_j^2\right)p_0}{\left(\sum_{l=0}^{q-1}\prod_{j=0}^{l-1}m_j^2\right)p_0 + r}+r\right]}^{-1}\\
            &= {\left[\frac{\left(\prod_{j=0}^{q-1}m_j^2\right)p_0}{\left(\sum_{l=0}^{q-1}\prod_{j=0}^{l-1}m_j^2\right)p_0 + r}\right]} {\left[\frac{\left(\prod_{j=0}^{q-1}m_j^2\right)p_0+\left(\sum_{l=0}^{q-1}\prod_{j=0}^{l-1}m_j^2\right)p_0 + r}{\left(\sum_{l=0}^{q-1}\prod_{j=0}^{l-1}m_j^2\right)p_0 + r}\right]}^{-1}\\
            &= \frac{\left(\prod_{j=0}^{q-1}m_j^2\right)p_0}{\left(\sum_{l=0}^{q}\prod_{j=0}^{l-1}m_j^2\right)p_0 + r}.
  \end{split}
\end{equation*}
\end{proof}

We can extend this approach the analysis mean as well, meaning that the analysis mean at each step is a linear fractional function of the initial input variance and the initial input mean.
\begin{lemma}\label{lem:analysismean}
  The analysis at the $i$-th step is
\begin{equation}
    \xa_i = \frac{\left(\prod_{j=0}^{i-1}m_j\right)\left[\left(\sum_{l=0}^i y_l \prod_{j=0}^{l-1}m_j\right)p_0+  r x_0 \right]}{\left(\sum_{l=0}^{i}\prod_{j=0}^{l-1}m_j^2\right)p_0 + r}.
\end{equation}
\end{lemma}

\begin{proof}
  Clearly $\xa_0 = x_0 - \frac{p_0}{p_0 + r}(x_0 - y_0)=\frac{y_0 p_0 + r x_0}{p_0+r}$. We thus proceed by induction:
\begin{equation*}
\begin{split}
    \xa_{q-1} &= \frac{\left(\prod_{j=0}^{q-2}m_j\right)\left[\left(\sum_{l=0}^{q-1} y_l \prod_{j=0}^{l-1}m_j\right)p_0+  r, x_0\right]}{\left(\sum_{l=0}^{q-1}\prod_{j=0}^{l-1}m_j^2\right)p_0 + r}\\
    \xa_q &= x^\|f_q - k_q(\xf_q-y_q)\\
                &= (1-k_q)m_{q-1}\xa_{q-1}+k_q y_q\\
                &= \left[1-\frac{\left(\prod_{j=0}^{q-1}m_j^2\right)p_0}{\left(\sum_{l=0}^q\prod_{j=0}^{l-1}m_j^2\right)p_0 + r}\right] \frac{\left(\prod_{j=0}^{q-1}m_j\right)\left[\left(\sum_{l=0}^{q-1} y_l \prod_{j=0}^{l-1}m_j\right)p_0+  r x_0\right]}{\left(\sum_{l=0}^{q-1}\prod_{j=0}^{l-1}m_j^2\right)p_0 + r} + \frac{y_q\left(\prod_{j=0}^{q-1}m_j^2\right)p_0 }{\left(\sum_{l=0}^q\prod_{j=0}^{l-1}m_j^2\right)p_0 + r}\\
                &= \frac{\left[\left(\sum_{l=0}^q\prod_{j=0}^{l-1}m_j^2\right)p_0 + r\right]\left(\prod_{j=0}^{q-1}m_j\right)\left[\left(\sum_{l=0}^{q-1} y_l \prod_{j=0}^{l-1}m_j\right)p_0+  r x_0 + y_q\left(\prod_{j=0}^{q-1}m_j^2\right)p_0\right]}{\left[\left(\sum_{l=0}^q\prod_{j=0}^{l-1}m_j^2\right)p_0 + r\right]\left[\left(\sum_{l=0}^{q+1}\prod_{j=0}^{l-1}m_j^2\right)p_0 + r\right]}\\
    &= \frac{\left(\prod_{j=0}^{q-1}m_j\right)\left[\left(\sum_{l=0}^{q} y_l \prod_{j=0}^{l-1}m_j\right)p_0+  r x_0\right]}{\left(\sum_{l=0}^{q}\prod_{j=0}^{l-1}m_j^2\right)p_0 + r}.
    \end{split}
\end{equation*}
\end{proof}

We define the following three useful sequences:
\begin{equation}
\begin{split}
    M_i &= \prod_{j=0}^{i-1} m_j, \\
    S_i &= \sum_{l=0}^i \prod_{j=0}^{l-1} m_j^2 = \sum_{l=0}^i M_l^2, \\
    B_i &= \sum_{l=0}^i y_l \prod_{j=0}^{l-1} m_j = \sum_{l=0}^i M_l y_l.
    \end{split}
\end{equation}
Intuitively we can think of $M_i$ as the forward model propagator from the initial step $0$ to the current step $i$, $S_i$ as the cumulative model variance propagator to step $i$, and $B_i$ as the cumulative observation propagator to step $i$.

We thus write:
\begin{align}
    \pa_i &= \frac{M_i^2 p_0}{S_i p_0 + r}, \label{eq:skfpa}\\
    \xa_i &= \frac{M_i(B_i p_0 + r x_0)}{S_i p_0 + r}. \label{eq:skfxa}
\end{align}

Note again that we have assumed that we have a perfect non-trivial model (assumption~\ref{ass:pm}), therefore:
\begin{equation}
\begin{split}
    M_0 &= 1,\\
    S_0 &= 1,\\
    S_0 < S_1 < &\cdots < S_n,\\
    M_i^2 &\leq S_i.
\end{split}    
\end{equation}

\begin{remark}
Consider a dynamical system described by real valued initial value problem
\begin{align*}
  \*y' = \*f(\*y(t)),\quad t_0\leq t\leq t_f,\quad \*y(t_0)=\*y_0.
\end{align*}
Taking a forward Euler step in time,
\begin{align*}
  \*y_{i+1} = \*y_i + h_i \*f(\*y_i),
\end{align*}
a linearization of the model could then be written as
\begin{align*}
    \*M_i = \*I + h_i \*J(\*y(t_i)),
\end{align*}
where $\*J(\*y(t)) = \left.\frac{d \*f}{d \*y}\right\rvert_{\*y(t)}$ is  the Jacobian of $\*f$.
Bounded chaotic systems generally have the property that $\lVert \prod_{j=0}^\infty \*M_j \rVert \to \infty$, therefore the corresponding scalar case is of particular interest.
\end{remark}

\subsection{SKF convergence}

In the Bayesian approach to uncertainty quantification we seek to correctly describe our information about the truth.
In the language of the scalar Kalman filter, our information is described by the mean and variance of a normal distribution. Thus, we wish to both optimally describe the truth via the mean, and optimally describe our confidence in it, through the variance.

Thus, the ideal desired behavior for the scalar Kalman filter is for it to be an unbiased estimator of the truth, meaning that the expected value of the analysis tends towards the truth in the step limit,
\begin{align}
  \lim_{i\to\infty}\mathbb{E}[\xa_i - x^\|t_i] = 0,
\end{align}
and an unbiased estimator of the variance in that estimate, meaning that the variance in the mean tends towards our description of it,
\begin{align}
  \lim_{i\to\infty} \left[\Var\left(\xa_i - x^\|t_i\right)- \pa_i\right] = 0.
\end{align}

Using equation~\eqref{eq:skfpa} we look at the deviation of the analysis mean from the truth at some arbitrary step $i$:
\begin{equation}\label{eq:xamxt}
\begin{split}
  \xa_i - x_i^\|t &= \left[\frac{M_i(B_i p_0 + r x_0)}{S_i p_0 + r} - M_i \xt_0 \right] =  \frac{M_i \left[(B_i - \xt_0 S_i) p_0 + r (x_0 - \xt_0)\right]}{S_i p_0 + r}\\
                                       &= \frac{M_i r(x_0 - \xt_0)}{S_i p_0 + r} +p_0  \frac{M_i(B_i - \xt_0 S_i)}{S_i p_0 + r}.
\end{split}
\end{equation}
We therefore have to look at the asymptotic behavior of two terms. The first term is the ratio of model propagator to that of the variance propagator:
\begin{align}
  \frac{M_i}{S_i}.
\end{align}
The second term is the  propagated cumulative normalized observation deviation,
\begin{align}
   \frac{M_i(B_i - \xt_0 S_i)}{S_i}.\label{eq:cumulativenormalizedobservationdeviation}
\end{align}

\begin{lemma}
\label{lem:cummodelas}
  The cumulative model variance propagator grows faster than the model propagator:
 \begin{align}
    \lim_{i\to\infty} \frac{M_i}{S_i} = 0.
  \end{align}
\end{lemma}
\begin{proof}
  Without loss of generality it suffices to look at $\lvert M_i \rvert$. We will examine the following exhaustive list of cases:
  \begin{enumerate}
  \item $\lim_{i\to\infty}\lvert M_i \rvert = 0$,\label{en:pp1}
  \item $\lim_{i\to\infty}\lvert M_i \rvert = C>0$,\label{en:pp2}
  \item $\lim_{i\to\infty}\lvert M_i \rvert = \infty$,\label{en:pp3}
  \item $\lim_{i\to\infty}\lvert M_i \rvert$ does not exist.\label{en:pp4}
  \end{enumerate}

  For Case~\ref{en:pp1} it suffices to see that $S_i \geq 1$, thus
  \begin{align*}
    \lim_{i\to\infty}\frac{\lvert M_i \rvert}{S_i} \leq \lim_{i\to\infty}\frac{\lvert M_i \rvert}{1} = 0.
  \end{align*}

  For Case~\ref{en:pp2}, there exists a step, $q$, and $\delta$, such that $C + \delta \geq \lvert M_i \rvert \geq C - \delta > 0, \forall i>q$, therefore,
  \begin{align*}
    \lim_{i\to\infty}\frac{\lvert M_i \rvert}{S_i} \leq \lim_{i\to\infty}\frac{C+\delta}{(i-q){(C-\delta)}^2} = 0.
  \end{align*}
  One will note that this case also proves the case where $\lvert M_i \rvert$ is bounded but does not converge.

  For Case~\ref{en:pp3}, observe that
  \begin{align*}
    \lim_{i\to\infty}\frac{\lvert M_i \rvert}{S_i} &= \lim_{i\to\infty} \frac{\lvert M_i \rvert}{\sum_{l=0}^i M_l^2}\\
                                                   &\leq \lim_{i\to\infty}\frac{\lvert M_i \rvert}{M_i^2}\\
                                                   &=\lim_{i\to\infty}\frac{1}{\lvert M_i \rvert} = 0.
  \end{align*}

  For case~\ref{en:pp4}, observe that,
  \begin{equation}
      \inf_{1\leq j\leq i} M_j \leq M_i \leq \sup_{1\leq j\leq i} M_j,
  \end{equation}
  and as the two bounds fall into one of our other three categories, the estimates collapse, and we regress to the former.
\end{proof}

\begin{lemma}\label{lem:msqsi}
  The variance propagator is at least as large as the square of the model propagator,
  \begin{align}
    0 \leq \frac{M_i^2}{S_i} \leq 1.
  \end{align}
\end{lemma}
\begin{proof}
  This trivially follows from the definitions.
\end{proof}

\begin{lemma}\label{lem:podurv}
  The propagated cumulative normalized observation deviation is an unbiased random variable with variance converging to the observation variance times the ratio of the square of the model propagator to the variance propagator. In particular,
  \begin{align}
    \frac{M_i(B_i - \xt_0 S_i)}{S_i} &= \sum_{l=0}^i \varepsilon_{l,i},\\
    \mathbb{E}\left[\frac{M_i(B_i - \xt_0 S_i)}{S_i}\right] &= 0,\\
    \lim_{i\to\infty}\Var\left(\frac{M_i(B_i - \xt_0 S_i)}{S_i}\right) &= r \lim_{i\to\infty} \frac{M_i^2}{S_i},
  \end{align}
  where $\varepsilon_{l,i} \sim \!N\left(0, M_i^2 M_l^2 S_i^{-2} r\right)$.
\end{lemma}
\begin{proof}
   Every observation $y_l$ is a sample from the distribution $\!N(x_l^\|t,r)$. Define $y_{0,l} = M_l^{-1}y_l$, and observe that $\varepsilon_{l} = M_l(y_{0,l} - \xt_0) = y_l - \xt_l \sim \!N(0, r)$. Additionally define $\varepsilon_{l,i} = M_i M_l S_i^{-1}\varepsilon_{l} \sim \!N\left(0, M_i^2 M_l^2 S_i^{-2} r\right)$. Now we manipulate:
  \begin{align*}
    \frac{M_i(B_i - \xt_0 S_i)}{S_i} &= \frac{M_i}{S_i} \sum_{l=0}^i (M_l y_l - M_l^2 x_0^t) 
                                         = \frac{M_i}{S_i} \sum_{l=0}^i M_l^2(y_{0,l} - x_0^t) \\
                                       &= \frac{M_i}{S_i}\sum_{l=0}^i M_l\varepsilon_{l}
                                         = \sum_{l=0}^i \varepsilon_{l,i}.
  \end{align*}
  As for the expected value and variance, 
  \begin{align*}
    \mathbb{E}\left[\frac{M_i(B_i - \xt_0 S_i)}{S_i}\right] = \mathbb{E}\left[\sum_{l=0}^i \varepsilon_{l,i}\right] =  \sum_{l=0}^i \mathbb{E}\left[\varepsilon_{l,i}\right] = 0,\\
     \lim_{i\to\infty} \Var\left(\sum_{l=0}^i \varepsilon_{l,i} \right) = r \lim_{i\to\infty} \frac{M_i^2}{S_i^2} \sum_{l=0}^i M_l^2 = r \lim_{i\to\infty} \frac{M_i^2}{S_i^2} S_i = r \lim_{i\to\infty} \frac{M_i^2}{S_i},
  \end{align*}
  as required.
\end{proof}

\begin{corollary}\label{cor:weakisanalysisvar}
  In the step limit, the analysis uncertainty estimate, $\pa_i$, approaches the variance of the propagated cumulative normalized observation deviation,
  \begin{align}
    \lim_{i\to\infty} \pa_i = \lim_{i\to\infty} \Var\left(\frac{M_i(B_i - \xt_0 S_i)}{S_i}\right).
  \end{align}
\end{corollary}
\begin{proof}
  \begin{align*}
    \lim_{i\to\infty} \pa_i = r \lim_{i\to\infty} \frac{M_i^2 p_0}{S_i p_0 + r} = r \lim_{i\to\infty} \frac{M_i^2}{S_i} = \lim_{i\to\infty} \Var\left(\sum_{l=0}^i \varepsilon_{l,i} \right) = \lim_{i\to\infty} \Var\left(\frac{M_i(B_i - \xt_0 S_i)}{S_i}\right).
  \end{align*}
\end{proof}

Note that the analysis variance is not zero in the step limit for models that grow sufficiently fast. Take $M_i^2 = e^i$, then,
\begin{align*}
  \lim_{i\to\infty} \frac{M_i^2}{S_i} =   \lim_{i\to\infty} \frac{e^i}{\sum_{l=0}^i e^l} = \lim_{i\to\infty}\frac{e^i - e^{i+1}}{1-e^{i+1}} = \frac{e - 1}{e} > 0.
\end{align*}
As a consequence of this, we can therefore have non-zero uncertainty in the analysis in the step limit, even for perfect models!

\begin{theorem}
  In the step limit, the mean of the scalar Kalman filter approaches the truth, and our description of the variance tends towards the variance in the mean, meaning that, 
    \begin{align}
    \lim_{i\to\infty} \mathbb{E}[\xa_i - x_i^\|t] &= 0,\\
    \lim_{i\to\infty} \left[\Var(\xa_i - x_i^\|t) - \pa_i\right] &= 0.
  \end{align}
\end{theorem}
\begin{proof}
  Following equation~\eqref{eq:xamxt}, we manipulate:
  \begin{align*}
    \lim_{i\to\infty}\xa_i - x_i^\|t &= \lim_{i\to\infty} \left[\frac{M_i(B_i p_0 + r x_0)}{S_i p_0 + r} - M_i \xt_0 \right] = \lim_{i\to\infty} \frac{M_i \left[(B_i - \xt_0 S_i) p_0 + r (x_0 - \xt_0)\right]}{S_i p_0 + r}\\
                                       &= \left[\lim_{i\to\infty} \frac{M_i r(x_0 - \xt_0)}{S_i p_0 + r}\right] + \left[p_0 \lim_{i\to\infty} \frac{M_i(B_i - \xt_0 S_i)}{S_i p_0 + r}\right].
  \end{align*}
  The term $\lim_{i\to\infty} \frac{M_i r(x_0 - \xt_0)}{S_i p_0 + r}$ always converges to zero in the limit by Lemma~\ref{lem:cummodelas}. As for $p_0 \lim_{i\to\infty} \frac{M_i(B_i - \xt_0 S_i)}{S_i p_0 + r}$, its expected value is zero by Lemma~\ref{lem:podurv}, and its variance is $\lim_{i\to\infty} \pa_i$ by Corollary~\ref{cor:weakisanalysisvar}. 
\end{proof}

This shows that in the step limit, the scalar Kalman filter description of the moments converges to the moments describing the uncertainty.

\section{The Scalar Pedagogical Ensemble Kalman Filter (SPEnKF)}%
\label{sec:spenkf}

\subsection{Definition of the SPEnKF}

What is the fundamental characteristic that defines the ensemble Kalman filter?  We argue that the key component is the non-linear expression used to build the sampled covariance estimation, and seek to create the simplest possible version of the EnKF which still carries with it uncertain information from sampling the (co-)variance. 

\begin{assumption}[Identical initial sampling]
We assume that now our two inputs are $\hat{x}^\|a_0\coloneqq x_0$ (the same mean input as to that of the scalar Kalman filter), and $\*a^\|a_0 = \*a$ the vector of $N$ anomalies about the mean, such that ${[\*a]}_{1\leq i\leq N}\sim \!N(0,p_0)$ (anomalies are sampled exactly from a distribution with the variance used by the exact scalar Kalman filter).
\end{assumption}

\begin{lemma}
If ${[\*a]}_{1\leq i\leq N}\sim \!N(0,p_0)$, is a collection of $N$ samples from the distribution, then 
\begin{equation}
\label{eq:gamma}
\frac{1}{N}\,(\*a\cdot \*a) = \frac{1}{N}\,\sum_{i=1}^N a_i^2 \sim \Gamma\left(\frac{N}{2},\frac{N}{2 p_0}\right).
\end{equation}
\end{lemma}

\begin{proof}
Consider first the case of $\tilde{a}_i\sim~\!N(0,1)$, by the definition of the chi-square distribution, $\tilde{\*a}\cdot\tilde{\*a}\sim\chi^2_N=\Gamma\left(\frac{N}{2},\frac{1}{2}\right)$. Note also that if $x\sim\Gamma(\alpha,\beta)$, then $c x\sim\Gamma(\alpha,c^{-1}\beta)$, and that that $\sqrt{p_0}\,\tilde{a}_i=a_i\sim \!N(0,p_0)$. Therefore \eqref{eq:gamma} holds.
\end{proof}

In what follows we denote by ``hat'' the ensemble-estimated variances. For example, the initial sample variance for an over-sampled ensemble ($N > 1$) is 
\begin{equation}
\label{eq:hatp0}
\hat{p}_0=\frac{1}{N} (\*a \cdot \*a) \sim \Gamma\left(\alpha, \frac{\alpha}{p_0}\right), \quad \alpha := \frac{N}{2}.
\end{equation}

We will again assume a perfect model (assumption~\ref{ass:pm}) and a constant observation error variance (assumption~\ref{ass:constobs}).

We then construct the filter to as closely as possible approximate the behavior of the exact scalar filter. Propagating the mean one step:
\begin{equation}
\label{eq:spenkff}
\begin{split}
  \hat{x}^\|f_{i+1} &= m_i \hat{x}^\|a_i,\\
  \hat{x}^\|a_i &= \hat{x}^\|f_i + \hat{k}_i(y_i - \hat{x}^\|f_i),
  \end{split}
\end{equation}
is exactly the same as in the scalar case, with the exception of the Kalman gain, which is dependent on the anomalies. Propagating the anomalies forward one step therefore works as follows:
\begin{subequations}
\label{eq:spenkfa}
\begin{align}
  \*a^\|f_{i} &= m_{i-1} \*a^\|a_{i-1},\\
  \*a^\|a_i &= {\left(\hat{p}^\|a_i\right)}^{\frac{1}{2}}{\left(\hat{p}^\|f_i\right)}^{-\frac{1}{2}}\*a^\|f_i,\\
  \hat{k}_i &= \frac{\hat{p}^\|f_i}{\hat{p}^\|f_i+r},\\
  \hat{p}^\|f_i &= \frac{1}{N} (\*a^\|f_i \cdot \*a^\|f_i),\label{eq:spenkfpf}\\
  \hat{p}^\|a_i &= (1-\hat{k}_i)\hat{p}^\|f_i. \label{eq:spenkfpa}
\end{align}
\end{subequations}
Here the transformation ${\left(\hat{p}^\|a_i\right)}^{\frac{1}{2}}{\left(\hat{p}^\|f_i\right)}^{-\frac{1}{2}}$ would be the optimal transport in the case where it was assumed that ${\left[a^\|f_i\right]}_j \sim \!N(0, \hat{p}^\|f_i)$ and ${\left[a^\|a_i\right]}_j \sim \!N(0, \hat{p}^\|a_i)$. This is however not the case.

\subsection{Properties of SPEnKF}

\begin{lemma}
  The analysis variance, $\hat{p}^\|a_i$, computed by \eqref{eq:spenkfpa}, is exactly the sampled variance matrix of the analysis anomalies:
  \begin{align}
    \hat{p}^\|a_i &= \frac{1}{N} (\*a^\|a_i \cdot \*a^\|a_i),
  \end{align}
  and the forecast variance, $\hat{p}^\|f$, computed by \eqref{eq:spenkfpf}, at step $i+1$ is exactly the previous analysis variance propagated by the model:
  \begin{align}
    \hat{p}^\|f_{i+1} = m_i^2\hat{p}^\|a_i.
  \end{align}
\end{lemma}
\begin{proof}
  By simple manipulation of \eqref{eq:spenkff} and \eqref{eq:spenkfa}:
  \begin{align*}
    \frac{1}{N} (\*a^\|a_i \cdot \*a^\|a_i) &= \left(\frac{(1-\hat{k}_i)\hat{p}^\|f_i}{\hat{p}^\|f_i}\right)\hat{p}^\|f_i = (1-\hat{k}_i)\hat{p}^\|f_i = \hat{p}^\|a_i,\\
    \hat{p}^\|f_{i+1} &= \frac{1}{N} (\*a^\|f_{i+1} \cdot \*a^\|f_{i+1}) = m_i^2 \frac{1}{N} (\*a^\|a_i \cdot \*a^\|a_i) = m_i^2 \hat{p}^\|a_i.
  \end{align*}
\end{proof}
This implies that the underlying anomalies are not important to the resulting distribution after several steps of the algorithm. All that matters to determining the resulting distribution, and thus the information of the variance at step $i$ is the distribution of the initial variance estimate at the onset.
The problem therefore reduces from attempting to grasp the distribution of the anomalies at a certain step---which almost certainly is not normal and whose members are not independent---to one of looking at a simple scalar.

\begin{lemma}
  The Kalman gain $\hat{k}_i$ of the SPEnKF is a random variable of the form $\frac{a_i \hat{p}_0}{c_i \hat{p}_0 + d_i}$, where $\hat{p}_0$ is distributed according to \eqref{eq:hatp0}.
\end{lemma}
\begin{proof}
  As the evolution of the variance in the SPEnKF is identical to that of the exact scalar Kalman filter, by Lemma~\ref{lem:kalmangain},
  \begin{align*}
    \hat{k}_i = \frac{a_i \hat{p}_0}{c_i \hat{p} + d_i},\\
    a_i = M_i^2,\quad c_i = S_i,\quad d_i = r,
  \end{align*}
  as required.
\end{proof}

\begin{lemma}
  The analysis mean $\hat{x}^\|a_i$ of the SPEnKF is a random variable of the form $\frac{a_i \hat{p}_0 + b_i}{c_i \hat{p}_0 + d_i}$,where $\hat{p}_0$ is distributed according to \eqref{eq:hatp0}.
\end{lemma}
\begin{proof}
  As the analysis mean evolves with the same exact principles as in the canonical exact Kalman filter, Lemma~\ref{lem:analysismean} applies, and as such,
  \begin{gather*}
    \hat{x}^\|a_i = \frac{a_i \hat{p}_0 +  b_i}{c_i\hat{p}_0 + d_i},\\
    a_i = M_i B_i,\quad b_i = M_i x_0 r,\quad c_i = S_i,\quad d_i= r.
  \end{gather*}
\end{proof}

This algorithm is obviously very similar, but not equivalent to the canonical scalar Kalman filter.

\subsection{Analysis of the perturbed problem}

As we have proven that the scalar Kalman filter (with a perfect non-trivial model) moments converge to the actual moments inherent in the estimates in the step limit, it suffices for us to prove that SPEnKF converges to the scalar Kalman filter in some certain asymptotic, and finite cases. We will accomplish this by showing degeneracy of the resulting distribution of the differences between the first two moment estimates of the SPEnKF and the SKF.

\begin{assumption}[Perturbed problem]
Let the SPEnKF take the inexact perturbed inputs $\tilde{p}_0$ (resulting from some perturbed anomalies), and $\tilde{x}_0$, whilst the corresponding exact scalar Kalman filter takes the unperturbed inputs $p_0$ and $x_0$.
\end{assumption} 

We now look at the discrepancy between the SPEnKF and the exact scalar KF. The discrepancy in the analysis variance, and analysis mean, at the $i$th step are random variables such that:
\begin{equation}\begin{split}
  \Delta p_i &= \hat{\tilde{p}}^a_i - p^a_i\\
             &= \frac{M_i^2 r\hat{\tilde{p}}_0}{S_i \hat{\tilde{p}}_0 + r} - \frac{M_i^2 r p_0}{S_i p_0 + r}\\
             &= \frac{M_i^2 r^2 \hat{\tilde{p}}_0 - M_i^2 r^2 p_0}{S_i(S_i p_0 + r) \hat{\tilde{p}}_0 + r(S_i p_0 + r)},\\
  \Delta x_i &= \hat{\tilde{x}}^a_i - x^a_i\\
             &= \frac{M_i B_i \hat{\tilde{p}}_0 + M_i r \tilde{x}_0}{S_i \hat{\tilde{p}}_0 + r} - \frac{M_i B_i p_0 + M_i r x_0}{S_i p_0 + r}\\
             &= \frac{M_i r(B_i - S_i x_0)\hat{\tilde{p}}_0 + M_i r(S_i p_0 \tilde{x}_0 + r \tilde{x}_0 - B_i p_0 - r x_0)}{S_i(S_i p_0 + r) \hat{\tilde{p}}_0 + r(S_i p_0 + r)}.
  \end{split}
    \end{equation}
    
Denote the generalized exponential integral function by:
\[
\Ei_n(z) := \int_1^\infty \frac{e^{- z t}}{t^n}\diff{t}.
\] 
By Lemma~\ref{lem:expandvar} we have that:
\begin{align}
  \mathbb{E}[\Delta p_i] &= \frac{\alpha M_i^2 r^2 e^{\frac{\alpha r}{S_i \tilde{p}_0}}}{S_i (S_i p_0 + r)}\left[-\frac{p_0}{\tilde{p}_0}\Ei_{\alpha}\left(\frac{\alpha r}{S_i \tilde{p}_0}\right) + \Ei_{\alpha+1}\left(\frac{\alpha r}{S_i \tilde{p}_0}\right)\right],\label{eq:pppm}\\
  \mathbb{E}[\Delta x_i] &= \frac{\alpha M_i r e^{\frac{\alpha r}{S_i \tilde{p}_0}}}{S_i (S_i p_0 + r)}\left[\frac{r (\tilde{x}_0 - x_0) - (B_i - \tilde{x}_0 S_i ) p_0 }{\tilde{p}_0}\Ei_{\alpha}\left(\frac{\alpha r}{S_i \tilde{p}_0}\right) + (B_i - S_i x_0) \Ei_{\alpha+1}\left(\frac{\alpha r}{S_i \tilde{p}_0}\right)\right],\label{eq:ppxm}\\
  \mathbb{E}[\Delta p_i^2] &= \frac{\alpha M_i^4 r^4 e^{\frac{\alpha r}{S_i \tilde{p}_0}}}{S_i^2 {(S_i p_0 + r)}^2} \left[\begin{aligned}\phantom{+}&\frac{\alpha p_0^2}{\tilde{p}_0^2} \Ei_{\alpha-1}\left(\frac{\alpha r}{S_i \tilde{p}_0}\right)
      + \frac{\alpha p_0 (2 \tilde{p}_0 - p_0)}{\tilde{p}_0^2}\Ei_{\alpha}\left(\frac{\alpha r}{S_i \tilde{p}_0}\right)\\
      +& \frac{\tilde{p}_0 + \alpha (\tilde{p}_0 - 2 p_0)}{\tilde{p}_0}\Ei_{\alpha+1}\left(\frac{\alpha r}{S_i \tilde{p}_0}\right)
      -(\alpha + 1) \Ei_{\alpha+2}\left(\frac{\alpha r}{S_i \tilde{p}_0}\right)\end{aligned}\right],\label{eq:pppv}\\
      \begin{split}
  \mathbb{E}[\Delta x_i^2] &= \frac{\alpha  r^2 M_i^2 e^{\frac{\alpha  r}{S_i \tilde{p}_0}}}{ S_i^2 \tilde{p}_0^2 {\left(S_i p_0 + r \right)}^2}\\
                             &\phantom{=}\left[
                             \begin{aligned}        
                               \phantom{+}&\alpha  {\left((B_i-\tilde{x}_0 S_i)p_0 + r(x_0 - \tilde{x}_0)\right)}^2 \Ei_{\alpha -1}\left(\frac{\alpha r}{S_i \tilde{p}_0}\right)\\
                               +& \alpha \left((B_i-\tilde{x}_0 S_i)p_0 + r(x_0 - \tilde{x}_0)\right) \left[\begin{aligned}
                                   -&(p_0 + 2\tilde{p}_0)\left(B_i - \frac{p_0\tilde{x}_0 + 2\tilde{p}_0 x_0}{p_0 + 2\tilde{p}_0} S_i\right)\\
                                   +&r\left(\tilde{x}_0-x_0\right)
                                 \end{aligned}
                               \right] \Ei_{\alpha }\left(\frac{\alpha r}{S_i \tilde{p}_0}\right)\\
                               +&\tilde{p}_0 \left(B_i-x_0 S_i\right) \left[\begin{aligned}
                                   \phantom{+}&(2\alpha p_0 + (\alpha + 1) \tilde{p}_0)\left(B_i - \frac{2\alpha\tilde{x}_0 p_0 + (\alpha + 1)x_0 \tilde{p}_0}{2\alpha p_0 + (\alpha + 1) \tilde{p}_0}S_i\right)\\
                                   +&2\alpha r(x_0 - \tilde{x}_0)
                                 \end{aligned}
                               \right] \Ei_{\alpha +1}\left(\frac{\alpha r}{S_i \tilde{p}_0}\right)\\
                               -&(\alpha +1) \tilde{p}_0^2{\left(B_i-x_0 S_i\right)}^2 \Ei_{\alpha +2}\left(\frac{\alpha r}{\tilde{p}_0 S_i}\right) 
                             \end{aligned}
                                  \right].\end{split}\label{eq:ppxv}
\end{align}

We first show that for certain special cases both the mean and the variance of the discrepancy approach zero, then we would show degeneracy of the perturbed problem.

\subsection{Convergence of the SPEnKF to the scalar KF as ensemble size grows to infinity}%
\label{subsec:ppfe}

The first case that we can look at is the one of the limiting case of the ensemble size growing to infinity. For the two algorithms to converge in ensemble size, their initial inputs have to be identical, as the algorithms operating on arbitrarily different inputs would necessitate arbitrarily different output.

\begin{theorem}\label{thm:enslimvar}
  When the initial inputs to the scalar Kalman filter and the SPEnKF are identical, $\tilde{x}_0=x_0$, $\tilde{p}_0 = p_0$, then for all steps $i$, in the  limit of ensemble size $\alpha\to\infty$, the expected value and variance of the discrepancy in the analysis variance are zero:
\begin{equation}\begin{split}
    \lim_{\alpha\to\infty} \mathbb{E}[\Delta p_i] &= 0\\
    \lim_{\alpha\to\infty} \left(\mathbb{E}[\Delta p_i^2] - \mathbb{E}{[\Delta p_i]}^2\right) &= 0.
\end{split}
  \end{equation}
\end{theorem}
\begin{proof}
by Corollary~\ref{cor:expectedconverge}, and Corollary~\ref{cor:varianceconverge},
\begin{equation}\begin{split}
    \lim_{\alpha\to\infty} \mathbb{E}[\Delta p_i] &= \frac{M_i^2 r^2 \tilde{p}_0 - M_i^2 r^2 p_0}{S_i(S_i p_0 + r) \tilde{p}_0 + r(S_i p_0 + r)}\\
                                                  &= \frac{M_i^2 r^2 p_0 - M_i^2 r^2 p_0}{{(S_i p_0 + r)}^2} = 0,\\
    \lim_{\alpha\to\infty} \left(\mathbb{E}[\Delta p_i^2] - \mathbb{E}{[\Delta p_i]}^2\right) &= \lim_{\alpha\to\infty} \mathbb{E}[\Delta p_i^2] = {\left(\frac{M_i^2 r^2 \tilde{p}_0 - M_i^2 r^2 p_0}{S_i(S_i p_0 + r) \tilde{p}_0 + r(S_i p_0 + r)}\right)}^2\\
                                                  &= {\left(\frac{M_i^2 r^2 p_0 - M_i^2 r^2 p_0}{{(S_i p_0 + r)}^2}\right)}^2 = 0.
\end{split}
  \end{equation}
  as required.
\end{proof}

\begin{theorem}\label{thm:enslimmean}
   When the initial inputs to the scalar Kalman filter and the SPEnKF are identical, $\tilde{x}_0=x_0$, $\tilde{p}_0 = p_0$,then for all steps $i$, in the  limit of ensemble size $\alpha\to\infty$, the expected value and variance of the discrepancy in the analysis mean are zero:
\begin{equation}\begin{split}
    \lim_{\alpha\to\infty} \mathbb{E}[\Delta x_i] &= 0\\
    \lim_{\alpha\to\infty} \left(\mathbb{E}[\Delta x_i^2] - \mathbb{E}{[\Delta x_i]}^2\right) &= 0,
\end{split}\end{equation}
\end{theorem}
\begin{proof}
by Corollary~\ref{cor:expectedconverge}, and Corollary~\ref{cor:varianceconverge},
\begin{equation}\begin{split}
    \lim_{\alpha\to\infty} \mathbb{E}[\Delta x_i] &= \frac{M_i r(B_i - S_i x_0)\tilde{p}_0 + M_i r(S_i p_0 \tilde{x}_0 + r \tilde{x}_0 - B_i p_0 - r x_0)}{S_i(S_i p_0 + r) \tilde{p}_0 + r(S_i p_0 + r)}\\
                                                  &= \frac{M_i r(B_i - S_i x_0)p_0 - M_i r(B_i - S_i x_0)p_0 + M_i r(r x_0 - r x_0)}{{(S_i p_0 + r)}^2} = 0,\\
    \lim_{\alpha\to\infty} \left(\mathbb{E}[\Delta x_i^2] - \mathbb{E}{[\Delta x_i]}^2\right) &= \lim_{\alpha\to\infty} \mathbb{E}[\Delta x_i^2] = {\left(\frac{M_i r(B_i - S_i x_0)\tilde{p}_0 + M_i r(S_i p_0 \tilde{x}_0 + r \tilde{x}_0 - B_i p_0 - r x_0)}{S_i(S_i p_0 + r) \tilde{p}_0 + r(S_i p_0 + r)}\right)}^2\\
    &= {\left(\frac{M_i r(B_i - S_i x_0)p_0 - M_i r(B_i - S_i x_0)p_0 + M_i r(r x_0 - r x_0)}{{(S_i p_0 + r)}^2}\right)}^2 = 0,
\end{split}\end{equation}
as required.
\end{proof}

Theorem~\ref{thm:enslimvar} and theorem~\ref{thm:enslimmean} combined show that in the asymptotic case of large ensemble sizes the trivial SPEnKF converges in means to the exact scalar Kalman filter and that the variances collapse to zero.

\subsection{Analysis of the perturbed problem in the case of a finite ensemble}
Arbitrarily large ensembles are theoretically nice, but impractical. Running the data assimilation scheme for an arbitrarily large number of steps however, is practical. Assume now that we have a finite over-sampled ensemble, $1 < \alpha < \infty$. 

Observe also that~\eqref{eq:ppxm} and~\eqref{eq:pppm},
\begin{equation}\begin{split}
  \lim_{i\to\infty} \mathbb{E}[\Delta p_i] &=
                                             \alpha r^2
                                             \left[\lim_{i\to\infty} \frac{M_i^2}{S_i}\right]
                                             \left[\lim_{i\to\infty} \frac{1}{S_i p_0 + r}\right]
                                             \left[\lim_{i\to\infty}e^{\frac{\alpha r}{S_i \tilde{p}_0}}\left(
                                           \Ei_{\alpha+1}\left(\frac{\alpha r}{S_i \tilde{p}_0}\right)
                                             -\frac{p_0}{\tilde{p}_0}\Ei_{\alpha}\left(\frac{\alpha r}{S_i \tilde{p}_0}\right) \right)\right]\\
  \lim_{i\to\infty} \mathbb{E}[\Delta x_i] &= \begin{aligned}\phantom{+}&
    \alpha r
    \left[\lim_{i\to\infty} e^{\frac{\alpha r}{S_i \tilde{p}_0}}\left(\frac{M_i(B_i - x_0 S_i)}{S_i(S_i p_0 + r)} \Ei_{\alpha+1}\left(\frac{\alpha r}{S_i \tilde{p}_0}\right) - \frac{p_0}{\tilde{p}_0} \frac{M_i(B_i - \tilde{x}_0 S_i)}{S_i(S_i p_0 + r)} \Ei_{\alpha}\left(\frac{\alpha r}{S_i \tilde{p}_0}\right)\right)\right]\\
    +&
    \frac{\alpha r^2 (\tilde{x}_0 - x_0)}{\tilde{p}_0}
    \left[\lim_{i\to\infty} \frac{M_i}{S_i p_0 + r}\right]
    \left[\lim_{i\to\infty} e^{\frac{\alpha r}{S_i \tilde{p}}} \Ei_{\alpha}\left(\frac{\alpha r}{S_i \tilde{p}_0}\right)\right]
  \end{aligned}
\end{split}\end{equation}

Note that in the terms above, the cumulative normalized observation deviation~\eqref{eq:cumulativenormalizedobservationdeviation}, is normalized by an additional $S_i$, meaning that we need to look at the cumulative doubly normalized observation deviation.

\begin{lemma}[Weak convergence of cumulative doubly normalized observation deviation]
\label{lem:weakconvpert}
  The cumulative doubly normalized observation deviation converges to zero in probability if the step limit, meaning that:
\begin{equation}\begin{split}
    \lim_{i\to\infty}\Pr\left[\left\lvert\frac{M_i(B_i - \xt_0 S_i)}{S_i^2}\right\rvert > \epsilon\right] = 0, \quad \forall \epsilon > 0,
\end{split}\end{equation}
  unconditionally on model behavior.
\end{lemma}
\begin{proof}
As in Lemma~\ref{lem:podurv}, observe that instead of $\varepsilon_{l,i}$, we deal with $S_i^{-1}\varepsilon_{l,i}$,
\begin{equation*}
  \begin{split}
    \frac{M_i(B_i - \xt_0 S_i)}{S_i^2} = \sum_{l=0}^i S_i^{-1}\varepsilon_{l,i},
\end{split}
  \end{equation*}
we require the variance of the mean to be zero,
\begin{equation*}\begin{split}
    \lim_{i\to\infty} \Var\left( \sum_{l=0}^i S_i^{-1}\varepsilon_{l,i}\right) = r\lim_{i\to\infty} \frac{M_i^2}{S_i^4} \sum_{l=0}^i M_l^2 = r \lim_{i\to\infty} \frac{M_i^2}{S_i^4} S_i = r\left[\lim_{i\to\infty} \frac{1}{S_i}\right]{\left[\lim_{i\to\infty} \frac{M_i}{S_i}\right]}^2 = 0,
\end{split}\end{equation*}
as required.
\end{proof}

\begin{lemma}[Strong convergence of cumulative doubly normalized observation deviation]\label{lem:strongconvpert}
 The cumulative doubly normalized observation deviation propagated forward by the model converges to zero almost surely,
\begin{equation}\begin{split}
    \Pr \left[\lim_{i\to\infty} \frac{M_i(B_i - \xt_0 S_i)}{S_i^2}\right] = 0,
\end{split}\end{equation}
conditionally, whenever $\limsup_{i\to\infty}\frac{i+1}{S_i} < \infty$.
\end{lemma}
\begin{proof}
The criteria for strong convergence are that
\begin{equation*}
  \begin{split}
    r \lim_{i\to\infty} \sup_{0\leq l \leq i} \left\{ \frac{M_i^2 M_l^2 {(i+1)}^2}{S_i^4}\right\} &< \infty,\\
    r \lim_{i\to\infty} \frac{M_i^2 {(i+1)}^2}{S_i^4} \sum_{l=0}^{i} \frac{M_l^2}{{(l+1)}^2} < \infty.
\end{split}
  \end{equation*}
The first condition ensures that the variances of all the individual random variables are finite, by stating that in the limit, their supremum is. The second condition is for the sufficient decay in their variances.
  
Note that $\frac{M_i^2}{S_i} \leq 1$ by Lemma~\ref{lem:msqsi} (moreover $\frac{M_l^2}{S_i}\leq 1$ for all $l\leq i$) therefore,
\begin{equation*}
\begin{gathered}
    r \lim_{i\to\infty} \sup_{0\leq l \leq i} \left\{ \frac{M_i^2 M_l^2 {(i+1)}^2}{S_i^4}\right\} \leq r \lim_{i\to\infty} \frac{{(i+1)}^2}{S_i^2} < \infty,\\
    r \lim_{i\to\infty} \frac{M_i^2 {(i+1)}^2}{S_i^4} \sum_{l=0}^{i} \frac{M_l^2}{{(l+1)}^2} \leq r \lim_{i\to\infty} \frac{{(i+1)}^2}{S_i^2} \sum_{l=0}^{i} \frac{1}{{(l+1)}^2}  < \infty.
\end{gathered}
\end{equation*}
as required.
\end{proof}

\begin{corollary}
 In the case of imperfect truth, when $\xt_0$ is replaced with some arbitrary constant $c$, and with slight abuse of notation,
\begin{equation}\begin{split}
    \lim_{i\to\infty} \frac{M_i(B_i - c S_i)}{S_i^2} = 0.
\end{split}\end{equation}
  in probability always or almost surely whenever$\limsup_{i\to\infty}\frac{i+1}{S_i} < \infty$.
\end{corollary}
\begin{proof}
\begin{equation*}\begin{split}
    \lim_{i\to\infty} \frac{M_i(B_i - c S_i)}{S_i^2}
    &= \lim_{i\to\infty} \frac{M_i(B_i - \xt_0 S_i) + M_i(\xt_0 S_i - c S_i)}{S_i^2}\\
    & = \lim_{i\to\infty} \frac{M_i(B_i - \xt_0 S_i)}{S_i^2} + \lim_{i\to\infty} \frac{M_i(\xt_0 S_i - c S_i)}{S_i^2}\\
    &= \lim_{i\to\infty} \frac{M_i(B_i - \xt_0 S_i)}{S_i^2} + \left[\lim_{i\to\infty} \frac{M_i}{S_i}\right] (\xt_0 - c)
      = 0.
\end{split}\end{equation*}
as required.
\end{proof}

\subsection{Optimal inflation factors}%
\label{subsec:opti}

From the form of \eqref{eq:pppm}, it can be surmised there exists a value of $\tilde{p}_0$ such that $\mathbb{E}[\Delta p_i]$ is zero for some particular value of $i$.

A natural thought is to find a multiplicative factor, $\theta$ such that $\tilde{p}_0 = \theta\, p_0$. In this context, $\theta$ is a heuristic multiplicative scaling factor that is applied to a covariance matrix, and is called inflation in the context of ensemble Kalman filters. We will use the term here to describe both initial (applied once at the beginning of the algorithm) and step-wise (applied at each step) scaling factors of our variances.

\begin{theorem}\label{thm:inflationinfinite}
There exists an initial inflation factor $\theta_{*}$ such that for the input variance value $\tilde{p}_0 = \theta_{*} p_0$, the expected value of the variance of the variance deviation in the perturbed problem, in the step limit, is zero, meaning that,
\begin{equation}
\lim_{i\to\infty}\mathbb{E}[\Delta p_i] = 0,
\end{equation}
which, from~\eqref{eq:pppm}, is equivalent to requiring that,
\begin{equation}\label{eq:thminfl}
\lim_{i\to\infty} e^{\frac{\alpha r}{S_i \tilde{p}_0}}\left( \Ei_{\alpha+1}\left(\frac{\alpha r}{S_i \tilde{p}_0}\right) - \frac{p_0}{\tilde{p}_0} \Ei_{\alpha}\left(\frac{\alpha r}{S_i \tilde{p}_0}\right)\right) = 0.
\end{equation}
\end{theorem}
\begin{proof}

  It is trivially evident that the solution to~\eqref{eq:thminfl} is the root of the function,
\begin{equation*}\begin{split}
    \!J(\theta) = \theta - \lim_{i\to\infty}\frac{\Ei_{\alpha}\left(\frac{\alpha r}{S_i \theta p_0}\right)}{\Ei_{\alpha+1}\left(\frac{\alpha r}{S_i \theta p_0}\right)}
\end{split}\end{equation*}

  There are only two cases, $S_i \to \infty$ and $S_i\to S_\infty < \infty$, as $S_i$ is a strictly monotonically increasing sequence.
  When $S_i\to\infty$
  
\begin{equation*}\begin{split}
    \theta - \lim_{i\to\infty}\frac{\Ei_{\alpha}\left(\frac{\alpha r}{S_i \theta p_0}\right)}{\Ei_{\alpha+1}\left(\frac{\alpha r}{S_i \theta p_0}\right)} = \theta - \frac{\Ei_{\alpha}(0)}{\Ei_{\alpha + 1 }(0)} = \theta - \frac{\alpha}{\alpha - 1},
\end{split}\end{equation*}
and the exact value for the inflation factor is
\[
\theta_{*} = \alpha\,{(\alpha - 1)}^{-1}.  
\]

In the case when $S_i\to S_\infty < \infty$, we have to find the root of the function:
\begin{equation*}\begin{split}
    \theta - \lim_{i\to\infty}\frac{\Ei_{\alpha}\left(\frac{\alpha r}{S_i \theta p_0}\right)}{\Ei_{\alpha+1}\left(\frac{\alpha r}{S_i \theta p_0}\right)} &= \theta - \frac{\Ei_{\alpha}\left(\frac{\alpha r}{S_\infty \theta p_0}\right)}{\Ei_{\alpha+1}\left(\frac{\alpha r}{S_\infty \theta p_0}\right)} = \theta - \frac{S_\infty \theta p_0}{\alpha r e^{\frac{\alpha r}{S_\infty \theta p_0}}\Ei_{\alpha+1}\left(\frac{\alpha r}{S_\infty \theta p_0}\right)} + \frac{S_\infty \theta p_0}{r}\\
  e^{\frac{\alpha r}{S_\infty \theta p_0}}\Ei_{\alpha+1}\left(\frac{\alpha r}{S_\infty \theta p_0}\right) &= \frac{S_\infty p_0}{\alpha(S_\infty p_0 + r)}
\end{split}\end{equation*}
Let $\mathfrak{E}_{\alpha+1}(z) = e^z \Ei_{\alpha+1}(z)$, and $\mathfrak{E}^{-1}_{\alpha+1}(z)$ be the corresponding inverse, which, as $\mathfrak{E}_{\alpha+1}(z)$ is a strictly monotonically decreasing function on $[0, \infty)$, is implicitly defined on $(0, \alpha^{-1}]$. As $0 < \frac{S_\infty p_0}{\alpha(S_\infty p_0 + r)} < \frac{1}{\alpha}$,
\begin{equation*}\begin{split}
    \theta_{*} = {\left[\frac{S_\infty p_0}{\alpha r}\mathfrak{E}_{\alpha+1}^{-1}\left(\frac{S_\infty p_0}{\alpha(S_\infty p_0 + r)}\right)\right]}^{-1},
\end{split}\end{equation*}
 is the unique inflation factor satisfying the criterion.
\end{proof}
Note that in the `interesting case', when $S_i\to\infty$, $\theta_{*}$ only depends on the size of the ensemble and not on the asymptotic model behavior!

Note also, that this implies that there exists an inflation factor, such that if it is applied at the beginning of the algorithm, the variance perturbation will be zero for a particular finite step $i$.
\begin{lemma}\label{lem:stepwiseinflation}
There exists a step-wise inflation factor, $\theta_i$ such that when $\tilde{p}_0 = \theta_i p_0$, the variance deviation at a particular step, $i$, is zero, equivalently, 
\begin{equation}\begin{split}
  e^{\frac{\alpha r}{S_i \tilde{p}_0}}\left( \Ei_{\alpha+1}\left(\frac{\alpha r}{S_i \tilde{p}_0}\right) 
    - \frac{p_0}{\tilde{p}_0} \Ei_{\alpha}\left(\frac{\alpha r}{S_i \tilde{p}_0}\right)\right) = 0.
\end{split}\end{equation}
\end{lemma}
\begin{proof}
By similar manipulation as in theorem~\ref{thm:inflationinfinite}, it is evident that
\begin{equation*}\begin{split}
    \theta_i = {\left[\frac{S_i p_0}{\alpha r}\mathfrak{E}_{\alpha+1}^{-1}\left(\frac{S_i p_0}{\alpha(S_i p_0 + r)}\right)\right]}^{-1},
\end{split}\end{equation*}
as required.
\end{proof}

Applying the optimal inflation factor $\theta_i$ for a particular step at the start of the algorithm is impractical as the whole algorithm would have to be re-run to the current step. In practically implemented ensemble-based methods, inflation is applied at every step, therefore we must generalize our approach to such a methodology.

\begin{lemma}
  The sequence of optimal step-wise initial inflation factors, defined by lemma~\ref{lem:stepwiseinflation},  is monotonically increasing and is bounded from below by one,
\begin{equation}\begin{split}
    1\leq \theta_i \leq \theta_{i+1}.
\end{split}\end{equation}
\end{lemma}
\begin{proof}
Define the function $\!E(z) = {\left[z\mathfrak{E}^{-1}_{\alpha + 1}\left(\frac{1}{z^{-1} + \alpha}\right)\right]}^{-1}$, noting that $\!E\left(\frac{S_i p_0}{\alpha r}\right) = \theta_i$, and observe that
\begin{equation*}\begin{split}
    \mathfrak{E}'_{\alpha + 1}(x) &= e^x \Ei_{\alpha+1}(x)\left(1+\frac{\alpha}{x}\right) - \frac{1}{x},\\
    \frac{\mathrm{d} \mathfrak{E}^{-1}}{\mathrm{d} z} &= \frac{1}{\mathfrak{E}'_{\alpha + 1}\left(\mathfrak{E}^{-1}_{\alpha+1}\left(\frac{1}{z^{-1} + \alpha}\right)\right)} = \frac{\mathfrak{E}_{\alpha+1}^{-1}\left(\frac{1}{z^{-1} + \alpha}\right)}{{(\alpha z + 1)}\left(\mathfrak{E}_{\alpha+1}^{-1}\left(\frac{1}{z^{-1} + \alpha}\right)(\alpha^2 z + z + \alpha) - \alpha z - 1\right)},\\
    \frac{\mathrm{d} \!E}{\mathrm{d}z} &= \mathfrak{E}_{\alpha+1}^{-1}\left(\frac{1}{z^{-1} + \alpha}\right)\left[1+\frac{z}{{(\alpha z + 1)}^3\left(\mathfrak{E}_{\alpha+1}^{-1}\left(\frac{1}{z^{-1} + \alpha}\right)(\alpha^2 z + z + \alpha) - \alpha z - 1\right)}\right],
\end{split}\end{equation*}
and also observe that as $\mathfrak{E}^{-1}_{\alpha + 1}$ is monotonically decreasing, then by known inequalities,
\begin{equation*}\begin{split}
    0 \leq \mathfrak{E}^{-1}_{\alpha + 1} \left(\frac{1}{z^{-1} + \alpha}\right) \leq z^{-1},
\end{split}\end{equation*}
  therefore
\begin{equation*}\begin{split}
    1 = \frac{\alpha r}{S_i p_0} \frac{S_i p_0}{\alpha r}  \leq {\left[\frac{S_i p_0}{\alpha r} \mathfrak{E}^{-1}_{\alpha + 1}\left(\frac{1}{\frac{\alpha r}{S_i p_0} + \alpha}\right)\right]}^{-1} = {\left[\frac{S_i p_0}{\alpha r} \mathfrak{E}^{-1}_{\alpha + 1}\left(\frac{S_i p_0}{\alpha(S_i p_0 + r)}\right)\right]}^{-1} = \theta_i.
\end{split}\end{equation*}
A sufficient condition on $\frac{\mathrm{d} \!E}{\mathrm{d}z} > 0$ is that $1 - \frac{z}{{(\alpha z + 1)}^4} > 0$ which is evidently true $\forall z > 0$ when $\alpha \geq 1$. Thus as $\frac{S_i p_0}{\alpha r}$ is a monotonically increasing sequence by the definition of $S_i$,
\begin{equation*}\begin{split}
    \!E\left(\frac{S_i p_0}{\alpha r}\right) = \theta_i \leq \theta_{i+1} = \!E\left(\frac{S_{i+1} p_0}{\alpha r}\right),
\end{split}\end{equation*}
as required.
\end{proof}

\begin{corollary}
  The initial optimal inflation factor is the upper bound and the limit of the sequence of optimal inflation factors,
\begin{equation}\begin{split}
      \theta_i &\leq \theta_{*},\\
      \lim_{i\to\infty} \theta_i &= \theta_{*}.
\end{split}\end{equation}
\end{corollary}

We are now ready to describe sequential step-wise inflation factors, that can be applied continuously one after the other, while keeping both the expected value of the deviation of the variance, \eqref{eq:pppm}, and the expected value of the deviation of the mean, \eqref{eq:ppxm}, zero for every step $i$.

\begin{theorem}\label{thm:inf}
The application of the sequential step-wise inflation factors,
\begin{equation}\begin{split}
    \phi_{i+1} = \frac{\theta_{i+1}(S_i\theta_i p_0+r)}{\theta_i(S_i\theta_{i+1} p_0+r)}\label{eq:optinfphi},
\end{split}\end{equation}
to the forecast variance at the $i+1$th step, $p_{i+1}^\|f \leftarrow \phi_{i+1}p_{i+1}^\|f$, for all $i$, with $\phi_0 = \theta_0$ being applied at the initial time, and the sequential  addition of the true step-wise correction factor
\begin{equation}\begin{split}
    \psi_{i+1} = \frac{M_{i+1} (B_i - S_i \tilde{x}_0)(\theta_{i+1} - \theta_i ) p_0 r  }{(S_i\theta_{i+1} p_0+r)(S_i\theta_i p_0+r)}\label{eq:optcorpsi},
\end{split}\end{equation}
the forecast mean at the $i+1$th step, $x_{i+1}^\|f \leftarrow \psi_{i+1} + x_{i+1}^\|f$, is equivalent to applying $\theta_{i+1}$ at the initial onset of the algorithm.
\end{theorem}
\begin{proof}
Assume $\pa_{i} = \frac{M_{i}\theta_{i}p_0}{S_{i} \theta_{i}p_0 + r}r$, and $\xa_i=\frac{M_i(B_i\theta_i p_0 + r\tilde{x}_0)}{S_i p_0 + r}$ then
\begin{equation*}\begin{split}
  \phi_{i+1}\pf_{i+1} &= \frac{\theta_{i+1}(S_i\theta_i p_0+r)}{\theta_i(S_i\theta_{i+1} p_0+r)} \frac{M_{i+1}\theta_{i}p_0}{S_{i} \theta_{i}p_0 + r}r = \frac{M_{i+1}\theta_{i+1}p_0}{S_{i} \theta_{i+1}p_0 + r},\\
  \psi_{i+1} + x_{i+1}^\|f &= \frac{M_{i+1} (B_i - S_i \tilde{x}_0)(\theta_{i+1} - \theta_i ) p_0 r  }{(S_i\theta_{i+1} p_0+r)(S_i\theta_i p_0+r)} + \frac{M_{i+1}(B_i\theta_i p_0 + r\tilde{x}_0)}{S_i \theta_i p_0 + r} = \frac{M_{i+1}(B_i\theta_{i+1} p_0 + r\tilde{x}_0)}{S_i \theta_{i+1}p_0 + r},
\end{split}\end{equation*}
as required.
\end{proof}

In this way, we boot-strap step-wise correct inflation factors for sequentially applied inflation.

\begin{corollary}
  The sequential step-wise inflation factors are bounded from below by 1, and from above by $\alpha{(\alpha - 1)}^{-1}$,
\begin{equation}\begin{split}
    1\leq\phi_{i+1}\leq \frac{\alpha}{\alpha - 1}.
\end{split}\end{equation}
\end{corollary}
\begin{proof}
For the lower bound, 
  \begin{equation*}\begin{split}
        \phi_{i+1} &= \frac{\theta_{i+1}(S_i\theta_i p_0+r)}{\theta_i(S_i\theta_{i+1} p_0+r)} = \frac{S_i\theta_i \theta_{i+1}p_0+r\theta_{i+1}}{S_i\theta_i\theta_{i+1} p_0+r\theta_i}\\
        &\geq \frac{S_i\theta_i \theta_{i+1}p_0+r\theta_{i+1}}{S_i\theta_i\theta_{i+1} p_0+r\theta_{i+1}} = 1.
  \end{split}\end{equation*}
For the upper bound,
  \begin{equation*}\begin{split}
        \phi_{i+1} &= \frac{\theta_{i+1}(S_i\theta_i p_0+r)}{\theta_i(S_i\theta_{i+1} p_0+r)} = \frac{S_i\theta_i \theta_{i+1}p_0+r\theta_{i+1}}{S_i\theta_i\theta_{i+1} p_0+r\theta_i}\\
        &\leq \frac{\theta_{i+1}}{\theta_i} \leq \theta_{*} \leq \frac{\alpha}{\alpha - 1},
  \end{split}\end{equation*}
as required.
\end{proof}

This means that there is concrete evidence for an inflation factor somewhere above one being applied sequentially, step-wise in various ensemble Kalman filters. Additionally, as applying the $\phi$ inflation, but ignoring the $\psi$ correction could potentially incur additional unbounded error, even if the sequence of corrections converges in probability in time, there is the potential for catastrophe, meaning that some time of sequential correction to the mean needs to be applied in Ensemble Kalman filtering. However, in a non-linear setting, the state is typically bounded, and therefore the absence of correction factors might dissipate in time (or be drowned out by the ensemble).

\subsection{Convergence of the perturbed problem in the case of a finite ensemble}%
\label{subsec:cfe}

We now have all the tools to prove the convergence of the SPEnKF in the case of a finite ensemble.

\begin{theorem}[Finite ensemble convergence of the analysis variance of the SPEnKF to that of the scalar KF]\label{thm:finvarconv}
  In the case of a finite ensemble $(\alpha < \infty)$,
  \begin{enumerate}
  \item $\mathbb{E}[\Delta p_i]$ converges to zero in the step limit $i\to \infty$, and is always zero when optimal sequential step-wise inflation~\eqref{eq:optinfphi} is applied at each step,
and
  \item $\mathbb{E}[\Delta p_i^2]$ converges to zero in the step limit.
  \end{enumerate}
\end{theorem}
\begin{proof}
For $\mathbb{E}[\Delta p_i]$, recall from~\eqref{eq:pppm} that,
\begin{equation}\begin{split}
      \mathbb{E}[\Delta p_i] &= \frac{\alpha M_i^2 r^2 e^{\frac{\alpha r}{S_i \tilde{p}_0}}}{S_i (S_i p_0 + r)}\left[-\frac{p_0}{\tilde{p}_0}\Ei_{\alpha}\left(\frac{\alpha r}{S_i \tilde{p}_0}\right) + \Ei_{\alpha+1}\left(\frac{\alpha r}{S_i \tilde{p}_0}\right)\right],
\end{split}\end{equation}
and observe that
\begin{equation*}\begin{split}
    \lim_{i\to\infty} \frac{M_i^2}{S_i}
\end{split}\end{equation*}
  is bounded by a constant above, but converges linearly to zero when $\lim_{i\to\infty} \lvert M_i\rvert < \infty$. Note that when $M_i\not\to0$, $S_i\to\infty$, and the term
\begin{equation*}\begin{split}
    \lim_{i\to\infty} \frac{1}{S_i p_0 + r}
\end{split}\end{equation*}
  trivially converges to zero, linearly. Note also that
\begin{equation*}\begin{split}
e^{\frac{\alpha r}{S_i \theta_i p_0}}\left(
                                           \Ei_{\alpha+1}\left(\frac{\alpha r}{S_i \theta_i p_0}\right)
                                             -\frac{p_0}{\theta_i p_0}\Ei_{\alpha}\left(\frac{\alpha r}{S_i \theta_i p_0}\right) \right) = 0.
\end{split}\end{equation*}
thus the last term converges to zero when optimal inflation is applied, at each step.

For $\mathbb{E}[\Delta p_i^2]$, \eqref{eq:pppv}, it is trivial to observe that either $\lim_{i\to\infty}\frac{M_i^4}{S_i^2} = 0$, or $\lim_{i\to\infty}\frac{1}{{(S_i p_0 + r)}^2} = 0$ (or both), with the remaining terms converging to constants.
\end{proof}

\begin{theorem}[Finite ensemble convergence of the analysis mean of the SPEnKF to that of the scalar KF]\label{thm:finmeanconv}
  In the case of a finite ensemble $(\alpha < \infty)$, the term $\mathbb{E}[\Delta x_i]$:
  \begin{enumerate}
      \item converges to zero weakly always in the step limit,
      \item converges strongly when $\limsup_{i\to\infty}\frac{i+1}{S_i} < \infty$ in the step limit, and
      \item is always zero when $\tilde{x}_0 = x_0$, optimal inflation~\eqref{eq:optinfphi} and optimal correction~\eqref{eq:optcorpsi} are applied,
  \end{enumerate}
 and $\mathbb{E}[\Delta x_i^2]$ converges to zero weakly always in the step limit.
\end{theorem}
\begin{proof}
For $\mathbb{E}[\Delta x_i]$, from~\eqref{eq:ppxm},
\begin{equation}\begin{split}
  \left[\lim_{i\to\infty} e^{\frac{\alpha r}{S_i \tilde{p}}}\left(\frac{M_i(B_i - x_0 S_i)}{S_i(S_i p_0 + r)} \Ei_{\alpha+1}\left(\frac{\alpha r}{S_i \tilde{p}_0}\right) - \frac{p}{\tilde{p}_0} \frac{M_i(B_i - \tilde{x}_0 S_i)}{S_i(S_i p_0 + r)} \Ei_{\alpha}\left(\frac{\alpha r}{S_i \tilde{p}_0}\right)\right)\right],
\end{split}\end{equation}
converges to zero in probability by Lemma~\ref{lem:weakconvpert}. If $\limsup_{i\to\infty}\frac{i+1}{S_i} < \infty$, by Lemma~\ref{lem:strongconvpert} this converges strongly to zero. The term
\begin{equation}\begin{split}
  \frac{\alpha r^2 (\tilde{x}_0 - x_0)}{\tilde{p}_0}
    \left[\lim_{i\to\infty} \frac{M_i}{S_i p_0 + r}\right]
  \left[\lim_{i\to\infty} e^{\frac{\alpha r}{S_i \tilde{p}}} \Ei_{\alpha}\left(\frac{\alpha r}{S_i \tilde{p}_0}\right)\right],
\end{split}\end{equation}
converges to zero in general as it has the term $\left[\lim_{i\to\infty} \frac{M_i}{S_i p_0 + r}\right]$, otherwise if $x_0=\tilde{x}_0$, with optimal sequential step-wise inflation and correction, the term is zero always by Theorem~\ref{thm:inf}.

For $\mathbb{E}[\Delta x_i^2]$, \eqref{eq:ppxv}, each term of the summation has two multiples of the term from Lemma~\ref{lem:weakconvpert}, thus converging to zero weakly always, and strongly if $\limsup_{i\to\infty}\frac{i+1}{S_i} < \infty$, by Lemma~\ref{lem:strongconvpert}.
\end{proof}

Theorems~\ref{thm:finvarconv} and~\ref{thm:finmeanconv} together show that there is strong evidence that the full ensemble Kalman filter can converge to the Kalman filter in expected value in the case of a finite ensemble, in finite time, provided that optimal corrections are made in the algorithm. Additionally we provide very strong evidence that sequential step-wise inflation, as performed in many flavours of the ensemble Kalman filter is not a heuristic, but in fact can be derived from the underlying distributions associated with it.

\subsection{SPEnKF with imaginary perturbations of observations}%
\label{subsec:po}

The idea of perturbed observations was first introduced in order to attempt to correct the ensemble Kalman filter~\cite{burgers1998analysis} from a statistical point of view under certain incorrect simlifications and assumptions. The wrongly assumed independence of the Kalman gain estimate from the anomalies and expected value of the Kalman gain estimate being the Kalman filter Kalman gain being just two. Augmenting the stochastic ensemble Kalman analysis update with a vector of `perturbed observations', $\*\Xi$, derived from the assumed distribution of the unbiased observation error, the update of the EnKF with perturbed observations, can be written as,
\begin{equation*}
  \*x^\|a = \*x^\|f - \*K (\*H \*x^\|f  + \*\Xi - \*y^o\*1^\intercal),
\end{equation*}
which we can decompose into the following two updates:
\begin{equation*}\begin{split}
  \bar{\*x}^\|a &= \bar{\*x}^\|f - \*K(\*H\bar{\*x}^\|f - \*y^o),\\
  \*A^\|a &= \*A^\|f - \*K(\*H\*A^\|f - \*\Xi),
\end{split}\end{equation*}
with the first just being the standard Kalman update, and the second being the unique stochastic EnKF anomaly update. In the scalar case we will again ignore $\*H$, as before and replace $\*A$ with $\*a$ and $\*\Xi$ with $\*\xi$.

In order to avoid difficulty with vector inner products, we will be looking at imaginary perturbed observations as a surrogate for true perturbed observations.  Empirical results suggest that this is a better filter than that with real perturbations, thus we can say with some confidence that results about this filter will be a lower bound for the full SPEnKF with perturbed observations, though a full analysis is, as of yet, not in our reach. Additionally we will not be looking at the asymptotic case of steps. Instead, we will be computing a perturbed observation update and a normal SPEnKF update on the SPEnKF forecast, and looking at the discrepancy between the two.

We therefore assume that in the analysis update below, $\*a_i^\|f$ was obtained with an ideal square-root filter, run from step 0 to step $i$, and that the $\*a_i^\|a$ that is obtained via the imaginary perturbed observation approach will be discarded in favor of another square root update. We will thus look at the update,
\begin{equation*}
  \*a_i^\|a = (1-\hat{k}_i) \*a_i^\|f + \mathrm{i}\hat{k}\*\xi_i,
\end{equation*}
where $\*\xi_i$ is an ensemble of $N$ samples from $\!N(0,r)$. We will also modify the analysis variance equation to account for complex conjugates, and observe:
\begin{equation*}\begin{split}
  \hat{p}_i^\|a &= \frac{1}{N} \left(\*a_i^\|a \cdot \overline{\*a_i^\|a}\right)\\
                &= {(1-\hat{k}_i)}^2 \hat{p}_i^\|f + \hat{k}_i^2 \hat{r}_i\\
                &= \hat{k}_i r + \hat{k}_i^2 (\hat{r}_i - r)
\end{split}\end{equation*}
Representing the realizations in terms of random variables, we will arrive at the fact that the random variable representing the new analysis update can be written in the form: $P = r K + K^2 (R - r)$. It can be trivially shown that $R \sim \Gamma\left(\alpha, \frac{\alpha}{r}\right)$, and thus $\mathbb{E}[R] = r$.
Looking at the moments of $P$, we manipulate:
\begin{equation*}\begin{split}
  \mathbb{E}[r K + K^2 (R - r)] &= \mathbb{E}[r K] + \mathbb{E}[K^2]\mathbb{E}[(R - r)],\\
                              &= r \mathbb{E}[K]\\
  \text{Cov}(r K,K^2 (R - r)) &= \mathbb{E}[(r K - r\mathbb{E}[K])(K^2 (R - r) - \mathbb{E}[K^2 (R - r)])]\\
                              &= \mathbb{E}[r K^3 (R - r) - r K^2 (R - r)\mathbb{E}[K]]\\
                              &= \mathbb{E}[r K^3 - r K^2 \mathbb{E}[K]]\mathbb{E}[(R - r)]\\
                              &= 0,\\
  \Var(r K + K^2 (R - r)) &= \Var(r K) + \Var(K^2 (R - r)) + 2\text{Cov}(r K, K^2 (R - r))\\
                              &= r^2 \Var(K) + \mathbb{E}[{(K^2 (R - r) - \mathbb{E}[K^2 (R - r)])}^2]\\
                              &= r^2 \Var(K) + \mathbb{E}[K^4]\mathbb{E}[{(R - r)}^2].
\end{split}\end{equation*}
Thus we see that the expected value of a perturbed observation filter is the same as of a perfect square root ensemble filter, however we do incur additional variance.

We can analyze this additional term, $\mathbb{E}[K^4]\mathbb{E}[{(R - r)}^2]$ in two different ways, in the asymptotic case of ensemble size, and in the step limit with a finite ensemble.

Note first that, without proof,
\begin{equation}\begin{split}
    \mathbb{E}[{(R_i - r)}^2] &= \frac{r^2}{\alpha},\\
  \mathbb{E}[K_i^4] &= \frac{M_i^8}{6 p^4 S_i^7}
                      \left[\begin{aligned}
                      \phantom{+}& p \left(p S_i \left(p S_i \left(6 p S_i+\alpha  \
(\alpha  (\alpha +7)+18) r\right)+(2 \alpha +9) \alpha ^2 \
                      r^2\right)+\alpha ^3r^3\right)\\
                      -& \frac{\alpha  r e^{\frac{\alpha  r}{p S_i}} \
\left((\alpha +3) p S_i \left((\alpha +2) p S_i \left((\alpha +1) p \
S_i+3 \alpha 
   r\right)+3 \alpha ^2 r^2\right)+\alpha ^3 r^3\right) E_{\alpha \
}\left(\frac{r \alpha }{p S_i}\right)}{S_i}
\end{aligned}\right].
\end{split}\end{equation}
The asymptotic case of ensemble size is by far the easiest:
\begin{equation}\begin{split}
  \lim_{\alpha\to\infty} \mathbb{E}[{(R_i - r)}^2] &= 0\\
   \lim_{\alpha\to\infty} \mathbb{E}[K_i^4] &= \frac{M_i^8 p_0 (2 r^2 - 3 S_i r p_0 + 3 S_i^2 p_0^2)}{3 S_i^3 {(S_i p_0 + r)}^3},
\end{split}\end{equation}
it is therefore the case that,
\begin{equation}\begin{split}
  \lim_{\alpha\to\infty} \mathbb{E}[K_i^4]\mathbb{E}[{(R_i - r)}^2] = 0.
\end{split}\end{equation}
This shows that there is significant evidence that in the asymptotic case of ensemble size, perturbed observation filters are as good as square-root filters. 

Let's now look at the case of  a finite ensemble in the step limit, and the worst case where $S_i$ grows roughly as fast as $M_i^2$,
\begin{equation}\begin{split}
  \lim_{i\to\infty}\mathbb{E}[K_i^4] = \text{const}.
\end{split}\end{equation}
This means that in the worst case, our variance has an additional constant term of $\frac{r^2}{\alpha}$, which can potentially be large. While we cannot claim that this will hold for non-imaginary perturbed observations, we postulate that this term is, in part responsible for some of the additional error that is seen in that type of filter compared to that of a square-root filter.

\section{Extending SPEnKF to Multivariate Case}
\label{sec:multivariate}

We will now attempt to extend the SPEnKF to a limited multivariate case. Assume now that we are looking at a multivariate state space, $\*x$ of size $n$, Assume additionally that we have a perfect model, whose step is represented by a matrix with independent action occurring in a constant basis throughout all time, that is, 
\begin{align}
    \*L_i = \*Z \*M_i \*Z^{-1},
\end{align}
with $\*M_i = \text{diag}(m_{i, 1}, \dots m_{i, n})$ being a diagonal matrix of real values, and $\*Z$ being any invertible constant matrix.

Let the initial input to our algorithm consist of a mean, $\bar{\*v}_0$, and a set of anomalies $\*B^\|f_0$ such that $\left[\*B^\|f_0\right]_{(:, {1\leq i\leq N})} \sim \!N(\*0, \*Z\*P_0\*Z^\intercal)$, where $\*P_0 = \text{diag}(p_{0, 1}, p_{0, 2}, \dots p_{0, n})$. Let all observations come from a normal distribution with a constant covariance matrix, $\*w_i \sim \!N(\*v_i^\|t, \*Z\*R\*Z^\intercal)$ with 
$\*R = \text{diag}(r_1, \dots r_n)$. Converting out of the linear basis, we get the familiar notation,
\begin{align}
\begin{split}
    \*x &= \*Z^{-1}\bar{\*v},\\
    \*A &= \*Z^{-1}\*B,\\
    \*y &= \*Z^{-1}\*w.
    \end{split}
\end{align}
Note that this directly implies that the observations in the basis are distributed as $\*y_i \sim \!N(\*x_i^\|t, \*R)$, and the anomalies in the basis are distributed like $\left[\*A^\|f_0\right]_{(:, {1\leq i\leq N})} \sim \!N(\*0, \*P_0)$.

Note that long-term dynamics can be written in the form
\begin{align}
    \prod_{j = 0}^i \*L_i = \*Z \left(\prod_{j=0}^i \*M_i\right) \*Z^{-1},
\end{align}
meaning that if we initialize the perfect square root EnKF in the model basis, we only have to look at independent model dynamics. 

The SPEnKF formulas, \eqref{eq:spenkff} and \eqref{eq:spenkfa}, for the mean of the $j$th member of state space in the basis at the $i$th time, become,
\begin{equation}
\begin{split}
  \bar{x}^\|f_{i+1,j} &= m_{i,j} \bar{x}^\|a_{i,j},\\
  \bar{x}^\|a_{i,j} &= \bar{x}^\|f_{i,j} + \hat{k}_i(y_{i,j} - \bar{x}^\|f_{i,j}),
  \end{split}
\end{equation}
and for the variances,
\begin{equation}
\begin{split}
  \*a^\|f_{i+1,j} &= m_{i,j} \*a^\|a_{i,j},\\
  \*a^\|a_{i,j} &= {\left(\hat{p}^\|a_{i,j}\right)}^{\frac{1}{2}}{\left(\hat{p}^\|f_{i,j}\right)}^{-\frac{1}{2}}\*a^\|f_{i,j},\\
  \hat{k}_{i,j} &= \frac{\hat{p}^\|f_{i,j}}{\hat{p}^\|f_{i,j}+r_j},\\
  \hat{p}^\|f_{i,j} &= \frac{1}{N} (\*a^\|f_{i,j} \cdot \*a^\|f_{i,j}),\\
  \hat{p}^\|a_{i,j} &= (1-\hat{k}_{i,j})\hat{p}^\|f_{i,j}.
\end{split}
\end{equation}
Writing the mean formulas in matrix notation, we get,
\begin{equation}
\begin{split}
  \bar{\*x}^\|f_{i+1,j} &= \*M_i \bar{\*x}^\|a_{i,j},\\
  \bar{\*x}^\|a_{i,j} &= \bar{\*x}^\|f_{i,j} + \hat{\*K}_i(\*y_{i,j} - \bar{\*x}^\|f_{i,j}),
  \end{split}
\end{equation}
and for the covariance,
\begin{equation}
\begin{split}
  \*A^\|f_{i+1} &= \*M_i \*A^\|a_{i},\\
  \*A^\|a_{i} &= {\left(\hat{\*P}^\|a_{i}\right)}^{\frac{1}{2}}{\left(\hat{\*P}^\|f_{i}\right)}^{-\frac{1}{2}}\*A^\|f_{i},\\
  \hat{\*K}_{i} &= \hat{\*P}^\|f_{i}{(\hat{\*P}^\|f_{i}+\*R)}^{-1},\\
  \hat{\*P}^\|f_{i} &= \*I \circ \frac{1}{N} (\*A^\|f_{i} \*A^{\|f,\intercal}_{i}),\\
  \hat{\*P}^\|a_{i} &= (\*I-\hat{\*K}_i)\hat{\*P}^\|f_{i}.
\end{split}
\end{equation}
Note that these are almost identical to the ESRF formulas, \eqref{eq:meansqrttransport} and \eqref{eq:anomalysqrttransport}. The only difference comes in the covariance tapering, in this case commonly referred to as Schur-product localization in DA literature.

\section{Conclusions}
\label{sec:conclusions}

We introduce a toy idealized EnKF variant named the SPEnKF, for Scalar Pedagogical EnKF, about which we prove several results.
We show the trivial result that in the limit of ensemble size, the SPEnKF degenerates to that of the scalar Kalman filter.
We show that in the step limit, and with a finite ensemble, the SPEnKF converges to that of the scalar Kalman filter, weakly always, and strongly for ``useful'' problems.

We derive optimal sequential step-wise variance inflation and mean correction factors such that the expected values of the SPEnKF outputs converge exactly to that of the scalar Kalman filter in finite time and with a finite ensemble.
We thus provide an alternative explanation for the need for inflation in ensemble-based methods: it is the required in order for the EnKF estimates to be useful in the realistic finite step finite ensemble case.

We then apply this framework to a scalar imaginary perturbed observations Kalman filter and show that in the case of a finite ensemble, we introduce an additional variance proportional to the square of the observation error variance compared to that of the vanilla SPEnKF.

Future work would try to naturally generalize these results to the multivariate case. We believe that it is possible to show that methods such as Schur-product localization are also required for similar reasons. Moreover, there is evidence~\cite{bickel2008regularized} to suggest that this might be doable in the undersampled case as well.

\begin{acknowledgments}
  This work was supported by awards AFOSR DDDAS FA9550--17--1--0015, AFOSR DDDAS 15RT1037, NSF CCF--1613905, NSF ACI--17097276, and by the Computational Science Laboratory at Virginia Tech.
\end{acknowledgments}

\section*{References}

\bibliography{biblio}

\newpage

\begin{appendices}

\section{Useful probability results}

We will now go through the probabilistic preliminaries that we require in order to tackle the finite-ensemble and finite-time convergence of the SPEnKF.

\begin{lemma}\label{lem:distgen}
  If $f_X(x)$ is the probability density function of a random variable, $X$, supported on $(0,\infty)$, then the probability density function of $Y=\frac{a X + b}{c X + d} = g(X)$, where $a,b,c,d\in\mathbb{R}$, $ad-bc \not= 0$, and with the simplifying assumptions that $c,d>0$, is
  \begin{equation}\begin{split}
    f_Y(y) = \frac{\lvert b c-a d \rvert f_X\left(\frac{d y-b}{a-c y}\right)}{{(a-c y)}^2}I_{(\ell_1,\ell_2)},
  \end{split}\end{equation}
  where $\ell_1=\min\{\frac{a}{c},\frac{b}{d}\}$, and $\ell_2=\max\{\frac{a}{c},\frac{b}{d}\}$.
\end{lemma}
\begin{proof}
  Observe that $X = \frac{d Y - b}{a - c Y} = g^{-1}(Y)$, additionally note the following known properties of probability distributions:
  \begin{equation*}\begin{split}
    f_Y(y) &= \left\lvert\frac{d}{d y} g^{-1}(y)\right\rvert f_X(g^{-1}(y))I_{g((0,\infty))}.
  \end{split}\end{equation*}
  We then manipulate:
  \begin{equation*}\begin{split}
    \left\lvert\frac{d}{d y} g^{-1}(y)\right\rvert &= \frac{\lvert b c - a d\rvert}{{(a - c y)}^2},\\
    f_Y(y) &= \frac{\lvert b c-a d \rvert f_X\left(\frac{d y-b}{a-c y}\right)}{{(a-c y)}^2}I_{g((0,\infty))}.
  \end{split}\end{equation*}
  We then only have to provide $g((0,\infty))$:
  \begin{equation*}\begin{split}
    \lim_{x\to0^+} \frac{a x + b}{c x + d} &= \frac{b}{d},\\
    \lim_{x\to\infty} \frac{a x + b}{c x + d} &= \frac{a}{c}.
  \end{split}\end{equation*}
  We do not know which of these values is greater (or even positive and negative), but we can say that $Y$ is therefore supported on the interval between them, as required.
\end{proof}

\begin{corollary}\label{cor:distgamma1}
  If $X\sim\Gamma\left(\alpha,\frac{\alpha}{p}\right)$ and $Y=\frac{a X + b}{c X + d}$ with $c,d>0$ and $a\not=0$ then
  \begin{equation}\begin{split}
    f_Y(y) = \frac{\lvert bc - ad\rvert {\left(\frac{\alpha (d y - b)}{p(a - c y)}\right)}^{\alpha}e^{-\frac{\alpha (d y - b)}{p(a - c y)}}}{{(d y - b)}{(a - c y)}\Gamma(\alpha)}.\label{eq:distgamma1}
  \end{split}\end{equation}
\end{corollary}
\begin{proof}
  Note that the pdf of $X$ is
  \begin{equation*}\begin{split}
    f_X(x) = \frac{{\left(\frac{\alpha}{p}\right)}^\alpha}{\Gamma(\alpha)}x^{\alpha-1}e^{-\frac{\alpha}{p} x},
  \end{split}\end{equation*}
  thus
  \begin{equation*}\begin{split}
    f_Y(y) &=\frac{ \lvert b c - a d\rvert {\left(\frac{\alpha}{p}\right)}^\alpha {\left(\frac{d y - b}{a - c y}\right)}^{\alpha-1}e^{-\frac{\alpha}{p} \frac{d y - b}{a - c y}}}{{(a - c y)}^2\Gamma(\alpha)}\\
    &= \frac{\lvert bc - ad\rvert {\left(\frac{\alpha (d y - b)}{p(a - c y)}\right)}^{\alpha}e^{-\frac{\alpha (d y - b)}{p(a - c y)}}}{{(d y - b)}{(a - c y)}\Gamma(\alpha)},
  \end{split}\end{equation*}
  as required.
\end{proof}

If $E_n(z) = \int_1^\infty \frac{e^{- z t}}{t^n}\diff{t}$ is the generalized exponential integral function, then 
\begin{equation}\begin{split}
  E_n(z) = \frac{z^{n-1}e^{-z}}{\Gamma(n)}\int_0^\infty \frac{t^{n-1}e^{-zt}}{t+1}\diff{t},\\
  n E_{n+1}(z) + z E_n(z) = e^{-z},\\
\frac{1}{z + n} < e^{z}E_n(z) \leq \frac{1}{z + n - 1}.
\end{split}\end{equation}
Additionally, as $n\to\infty$,
\begin{equation}\begin{split}
  E_n(\lambda n) \sim \frac{e^{-\lambda n}}{(\lambda + 1)n} \sum_{j=0}^\infty \frac{A_j(\lambda)}{{(\lambda+1)}^{2j}}\frac{1}{n^j},\label{eq:assEi}
\end{split}\end{equation}
where $A_0(\lambda) = A_1(\lambda) = 1$, and
\begin{equation*}\begin{split}
  A_{j+1}(\lambda) = (1-2\lambda j)A_j(\lambda)+\lambda(\lambda + 1)\frac{\diff{A_j(\lambda)}}{\diff{\lambda}},
\end{split}\end{equation*}
All these come from the very helpful~\cite{olver2010nist}.

\begin{lemma}\label{lem:expandvar}
 If $X\sim \Gamma\left(\alpha,\frac{\alpha}{p}\right)$ and $Y=\frac{a X + b}{c X + d}$, with $c,d>0$, and $a\not=0$ then 
 \begin{equation}\begin{split}
     \mathbb{E}[Y] &= \frac{\alpha}{c}e^{\frac{\alpha d}{c p}}\left[\frac{b}{p}E_\alpha\left(\frac{\alpha d}{c p}\right) + a E_{\alpha+1}\left(\frac{\alpha d}{c p}\right)\right],\\
     \mathbb{E}\left[{(Y-\mathbb{E}[Y])}^2\right] &= \begin{aligned}\phantom{+\,}&\frac{\alpha^2 b^2 e^{\frac{\alpha d}{c p}}}{c^2 p^2}E_{\alpha - 1}\left(\frac{\alpha d}{c p}\right)
                      +\frac{\alpha^2 b(2 a p - b)e^{\frac{\alpha d}{c p}}}{c^2 p^2}E_\alpha\left(\frac{\alpha d}{c p }\right)\\
                      +\,&\frac{\alpha a(\alpha a p + a p -2\alpha b)e^{\frac{\alpha d}{c p}}}{c^2 p}E_{\alpha + 1}\left(\frac{\alpha d}{c p}\right)
                      -\frac{\alpha(\alpha+1)a^2 e^{\frac{\alpha d}{c p}}}{c^2}E_{\alpha+2}\left(\frac{\alpha d}{c p}\right) - {\mathbb{E}[Y]}^2.\end{aligned}
 \end{split}\end{equation}
\end{lemma}
\begin{proof}
Note that $Y$ is supported on the interval $\left(\min\{\frac{a}{c},\frac{b}{d}\}, \max\{\frac{a}{c},\frac{b}{d}\}\right)$, thus by Corollary~\ref{cor:distgamma1} the pdf of Y is given by~\eqref{eq:distgamma1}. We will make the variable substitution $t = \frac{c}{d} x$,
and observe that 
\begin{equation*}\begin{split}
    x = \frac{d}{c} t,\ \diff{x} = \frac{d}{c} \diff{t},\\
    \lim_{x\to 0} t = 0,\ \lim_{x\to \infty} t = \infty.
\end{split}\end{equation*}
It is therefore the case that
  \begin{equation*}\begin{split}
    \mathbb{E}[Y] &= \int_0^\infty \frac{a x + b}{c x + d} f_X(x)\diff{x}\\
    &= \int_0^\infty \frac{(a x + b) {\left(\frac{\alpha x}{p}\right)}^\alpha e^{-\frac{\alpha x}{p}}}{x (c x + d)\Gamma(\alpha)}\diff{x}\\
    &= \frac{{\left(\frac{\alpha d}{c p}\right)}^\alpha}{\Gamma(\alpha)} \int_0^\infty \frac{(a x + b) {\left(\frac{c x}{d}\right)}^\alpha e^{-\frac{\alpha x}{p}}}{x (c x + d)}\diff{x}\\
    &=\frac{{\left(\frac{\alpha d}{c p}\right)}^\alpha}{\Gamma(\alpha)} \int_0^\infty \frac{\left(\frac{a}{c} t + \frac{b}{d}\right) {t}^{\alpha - 1} e^{-\frac{\alpha d}{c p}t}}{t + 1}\diff{t}\\
    &= \left[\frac{\alpha b {\left(\frac{\alpha d}{c p}\right)}^{\alpha-1}}{c p \Gamma(\alpha)}\int_0^\infty \frac{t^{\alpha-1} e^{-\frac{\alpha d}{c p} t}}{t + 1}\diff{t}\right] + \left[\frac{\alpha a {\left(\frac{\alpha d}{c p}\right)}^\alpha}{c\Gamma(\alpha + 1)} \int_0^\infty \frac{t^\alpha e^{-\frac{\alpha d}{c p} t}}{t + 1}\diff{t}\right]\\
    &= \frac{\alpha}{c}e^{\frac{\alpha d}{c p}}\left[\frac{b}{p}E_\alpha\left(\frac{\alpha d}{c p}\right) + a E_{\alpha+1}\left(\frac{\alpha d}{c p}\right)\right],
  \end{split}\end{equation*}
  as required. As for the variance, note that $\mathbb{E}\left[{(Y-\mathbb{E}[Y])}^2\right] = \mathbb{E}[Y^2]-\mathbb{E}{[Y]}^2$. We thus first manipulate:
  \begin{equation*}\begin{split}
    \mathbb{E}[Y^2] &= \int_0^\infty {(\frac{a x + b}{c x + d})}^2 f_X(x)\diff{x}\\
                  &= \int_0^\infty \frac{{(a x + b)}^2 {\left(\frac{\alpha x}{p}\right)}^\alpha e^{-\frac{\alpha x}{p}}}{x {(c x + d)}^2\Gamma(\alpha)}\diff{x}\\
                  &= \frac{{\left(\frac{\alpha d}{c p} \right)}^\alpha}{\Gamma(\alpha)}\int_0^\infty \frac{{(a d t + b c)}^2 t^{\alpha-1} e^{-\frac{\alpha d}{c p}t}}{c^2 d^2 {(t + 1)}^2}\diff{t}\\
                  &= \frac{{\left(\frac{\alpha d}{c p} \right)}^\alpha}{c^2 d^2 \Gamma(\alpha)}\begin{aligned}\Bigg[
                      &\left.-\frac{{(a d t + b c)}^2 t^{\alpha-1} e^{-\frac{\alpha d}{c p}t}}{t+1}\right\rvert_0^\infty\\
                      &-\int_0^\infty\frac{ (a d t+b c)(a d t (\alpha  d t-(\alpha +1) c p)+b c (c (p-\alpha  p)+\alpha  d t)) t^{\alpha -2}e^{-\frac{\alpha  d }{c p}t} }{c p (t+1)}\diff{t}\Bigg]
                    \end{aligned}\\
                  &= \frac{{\left(\frac{\alpha d}{c p} \right)}^\alpha}{c^3 d^2 p \Gamma(\alpha)}\int_0^\infty \frac{ (a d t+b c)(\alpha  (c p-d t) (a d t+b c)+c p (a d t-b c)) t^{\alpha -2}e^{-\frac{\alpha  d }{c p}t} }{t+1}\diff{t}\\
                  &= \frac{{\left(\frac{\alpha d}{c p} \right)}^\alpha}{c^3 d^2 p \Gamma(\alpha)}\begin{aligned}\Bigg[\phantom{+}\,&\int_0^\infty \frac{\left(\alpha  b^2 c^3 p-b^2 c^3 p\right) t^{\alpha -2}e^{-\frac{\alpha  d }{c p}t}}{t + 1}\diff{t}\\
                    +\,&\int_0^\infty \frac{\left(2 \alpha a b c^2 d p-\alpha  b^2 c^2 d\right)t^{\alpha -1} e^{-\frac{\alpha  d}{c p}t} }{t+1}\diff{t}\\
                    +\,&\int_0^\infty \frac{\left(\alpha a^2 c d^2 p+a^2 c d^2 p - 2\alpha a   b c d^2\right)t^{\alpha } e^{-\frac{\alpha  d}{c p}t}}{t + 1}\diff{t}\\
                    -\,&\int_0^\infty \frac{a^2 \alpha  d^3 t^{\alpha +1} e^{-\frac{\alpha  d }{c p}t}}{t + 1}\diff{t}
                    \Bigg]\end{aligned}\\
                    &= \begin{aligned}\phantom{+}&\frac{\alpha^2 b^2 e^{\frac{\alpha d}{c p}}}{c^2 p^2}E_{\alpha - 1}\left(\frac{\alpha d}{c p}\right)
                      +\frac{\alpha^2 b(2 a p - b)e^{\frac{\alpha d}{c p}}}{c^2 p^2}E_\alpha\left(\frac{\alpha d}{c p }\right)\\
                      +&\frac{\alpha a(\alpha a p + a p -2\alpha b)e^{\frac{\alpha d}{c p}}}{c^2 p}E_{\alpha + 1}\left(\frac{\alpha d}{c p}\right)
                      -\frac{\alpha(\alpha+1)a^2 e^{\frac{\alpha d}{c p}}}{c^2}E_{\alpha+2}\left(\frac{\alpha d}{c p}\right),\end{aligned}
  \end{split}\end{equation*}
  with the rest trivial.
\end{proof}

\begin{corollary}
  If $X\sim\Gamma\left(\alpha,\frac{\alpha}{p}\right)$ and $Y=\frac{a X + b}{c X + d}$ with $c,d>0$ and $a\not=0$, then, without proof,
  \begin{equation}\begin{split}
      \mathbb{E}[Y^4] = \frac{a^4 (\alpha +1) r^2}{6 \alpha  c^7 p^4} 
      \left[\begin{aligned}
      \phantom{-}&p \left(6 c^3 p^3+\alpha  (\alpha  (\alpha +7)+18) c^2 d p^2+\alpha ^2 (2 \alpha +9) c d^2 p+\alpha ^3 d^3\right)\\
      -&\frac{\alpha  d e^{\frac{\alpha  d}{c p}} \left(
      \begin{aligned}
      (\alpha +1) (\alpha +2) (\alpha +3) c^3 p^3 &+3 \alpha  (\alpha
   +2) (\alpha +3) c^2 d p^2\\&+3 \alpha ^2 (\alpha +3) c d^2 p+\alpha ^3 d^3
   \end{aligned}\right ) E_{\alpha }\left(\frac{d \alpha }{c p}\right)}{c}\end{aligned}\right]
  \end{split}.\end{equation}
\end{corollary}


\begin{corollary}\label{cor:expectedconverge}
  The asymptotic behavior of the expected value  is
  \begin{equation}\begin{split}
    \lim_{\alpha\to\infty} \mathbb{E}[Y] = \frac{a p + b}{c p + d},
  \end{split}\end{equation}
  and converges sublinearly in $\alpha$.
\end{corollary}
\begin{proof}
  First observe that
  \begin{equation*}\begin{split}
    \mathbb{E}[Y] &= \frac{\alpha}{c}e^{\frac{\alpha d}{c p}}\left[\frac{b}{p}E_\alpha\left(\frac{\alpha d}{c p}\right) + a E_{\alpha+1}\left(\frac{\alpha d}{c p}\right)\right]\\
                  &= \frac{a}{c} + \frac{\alpha(bc - ad)}{c^2 p} e^{\frac{\alpha d}{c p}} E_\alpha\left(\frac{\alpha d}{c p}\right)
  \end{split}.\end{equation*}
  By~\eqref{eq:assEi}, observe that
  \begin{equation*}\begin{split}
    \lim_{\alpha\to\infty} \mathbb{E}[Y] &= \frac{a}{c} + \lim_{\alpha\to\infty} \frac{\alpha(b c - a d)}{c^2 p} \left[\frac{1}{\alpha \left(\frac{d}{c p} + 1\right)} + \!O\left(\frac{1}{\alpha^2}\right) \right]\\
    &= \frac{a}{c} + \frac{b c - a d}{c(c p + d)}\\
    &= \frac{a p + b}{c p + d},
  \end{split}\end{equation*}
  as required. 
  
  For the convergence rate it suffices to show that $\alpha e^{\frac{\alpha d}{c p}} E_\alpha\left(\frac{\alpha d}{c p}\right)$ converges sublinearly to $\frac{cp}{cp + d}$ in $\alpha$:
  \begin{equation*}\begin{split}
    \lim_{\alpha\to\infty} \frac{(\alpha + 1) e^{\frac{(\alpha + 1) d}{c p}} E_{\alpha + 1}\left(\frac{(\alpha + 1) d}{c p}\right) - \frac{cp}{cp + d}}{\alpha e^{\frac{\alpha d}{c p}} E_\alpha\left(\frac{\alpha d}{c p}\right) - \frac{cp}{cp + d}}
    &= \lim_{\alpha\to\infty} \frac{\frac{cp}{cp + d} - \frac{cp}{cp + d} + \frac{c^3 p^3}{(\alpha + 1) {(c p + d)}^3} + \!O(\frac{1}{\alpha^2})}{\frac{cp}{cp + d} - \frac{cp}{cp + d} + \frac{c^3 p^3}{\alpha {(c p + d)}^3} + \!O(\frac{1}{\alpha^2})}\\
    &= \lim_{\alpha\to\infty} \frac{\alpha}{\alpha + 1} = 1,
  \end{split}\end{equation*}
  which is sublinear convergence.
\end{proof}

\begin{corollary}\label{cor:varianceconverge}
  The asymptotic behavior of the variance is
  \begin{equation}\begin{split}
    \lim_{\alpha\to\infty} \mathbb{E}\left[{\left(Y-\mathbb{E}[Y]\right)}^2\right] = 0
  \end{split}\end{equation}
\end{corollary}
\begin{proof}
  Note that
  \begin{equation*}\begin{split}
    \mathbb{E}[Y^2] &= \begin{aligned}\frac{\alpha b^2}{c d p} - \frac{\alpha(\alpha-1)b^2 e^{\frac{\alpha  d}{c p}}}{c d p}E_\alpha\left(\frac{\alpha d}{c p}\right) + \frac{\alpha^2 b (2 a p - b) e^{\frac{\alpha  d}{c p}}}{c^2 p^2}E_\alpha\left(\frac{\alpha d}{c p}\right)\\
      +\frac{a(\alpha a p + a p - 2\alpha b)}{c^2 p} - \frac{\alpha a d(\alpha a p + a p - 2\alpha b) e^{\frac{\alpha  d}{c p}}}{c^3 p^2} E_\alpha\left(\frac{\alpha d}{c p}\right)\\
      -\frac{\alpha a^2}{c^2} + \frac{\alpha a^2 d}{c^3 p} - \frac{\alpha a^2 d^2 e^{\frac{\alpha  d}{c p}}}{c^4 p^2}E_\alpha\left(\frac{\alpha d}{c p}\right)
    \end{aligned}\\
    &= \frac{c p \left(a^2 c d p + \alpha  {(b c - a d)}^2\right)-\alpha  (b c-a d) e^{\frac{\alpha  d}{c p}} E_{\alpha }\left(\frac{\alpha d}{c p}\right) (\alpha  (c p+d) (b c-a d)-c p (a d+b c))}{c^4 d p^2},
  \end{split}\end{equation*}
  then we see that
  \begin{equation*}\begin{split}
    \lim_{\alpha\to\infty}\mathbb{E}[Y^2] &= \lim_{\alpha\to\infty}\frac{p^2 (b c-a d) (a d+b c)+\alpha  d {(a p + b)}^2 (c p+d)}{\alpha  d {(c p+d)}^3} + \!O\left(\frac{1}{\alpha}\right)\\
                                          &= {\left(\frac{a p + b}{c p + d}\right)}^2,
  \end{split}\end{equation*}
  Meaning that
  \begin{equation*}\begin{split}
    \lim_{\alpha\to\infty}\mathbb{E}\left[{(Y-\mathbb{E}[Y])}^2\right] &= \lim_{\alpha\to\infty}\mathbb{E}[Y^2] - \lim_{\alpha\to\infty}\mathbb{E}{[Y]}^2 \\
    &= 0,
  \end{split}\end{equation*}
  as required.
\end{proof}


\end{appendices}

\end{document}

%% file: headerdre.tex
\def\one{\boldsymbol{1}}
\newcommand\1{\one}
\def\!#1{\mathcal{#1}}
\def\*#1{\boldsymbol{\mathbf{#1}}}
\def\|#1{\textnormal{#1}}

\def\diff#1{\mathrm{d}#1}

\newcommand{\En}[1]{\mathsf{E}_{#1}}

\def\diff#1{\mathrm{d}#1}

\def\mx{X}

\def\mxmean{\overline{\mx}}
\def\mxmeanf{\mxmean^{\|f}}
\def\mxmeana{\mxmean^{\|a}}
\def\mxt{\*\mx^{\|t}}
\def\mxs#1{\*\mx^{(#1)}}
\def\my{\*Y}

\def\P{\*P}
\def\Pf{\P^{\|f}}
\def\Pa{\P^{\|a}}

\def\Phf{\hat{\P}{\vphantom{\P}}^{\|f}}
\def\Pha{\hat{\P}{\vphantom{\P}}^{\|a}}
\def\M{\*M}
\def\H{\*H}
\def\Q{\*Q}
\def\R{\*R}
\def\K{\*K}
\def\Kh{\hat{\K}}
\def\I{\*I}

\def\A{\*A}
\def\Af{\A^{\|f}}
\def\AfT{\A^{\|f,\mathsf{T}}}
\def\Aa{\A^{\|a}}
\def\AfT{\A^{\|f,\mathsf{T}}}

\def\x{x}
\def\xt{\x^{\|t}}
\def\xf{\x^{\|f}}
\def\xa{\x^{\|a}}
\def\p{p}
\def\pf{\p^{\|f}}
\def\pa{\p^{\|a}}

\def\Ndist{\!N}

\def\Ei{E}

\DeclareMathOperator{\Var}{Var}